\documentclass{article}

\RequirePackage{amsthm,amsmath,amsfonts,amssymb}
\RequirePackage[numbers]{natbib}

\usepackage[blocks]{authblk}
\usepackage{amssymb,amsmath,amsthm,mathtools,bbm, enumitem}
\usepackage{footmisc}
\usepackage[latin1] {inputenc}
\usepackage{xcolor}
\usepackage{subcaption}
\usepackage{comment}
\usepackage{tikz,caption,float}
\usepackage[colorlinks=true,linkcolor=blue]{hyperref}
\usetikzlibrary{decorations.pathreplacing,angles,quotes}
  
\newtheorem{theorem}{Theorem}[section]
\newtheorem{lemma}[theorem]{Lemma}

\newtheorem{proposition}[theorem]{Proposition}
\newtheorem{corollary}[theorem]{Corollary}

\definecolor{bbm}{RGB}{51,153,0}
\definecolor{above}{RGB}{128,0,128}
\definecolor{below}{RGB}{102,0,204}
\definecolor{cascade}{RGB}{204,0,0}
\definecolor{iid}{RGB}{153,51,0}

\def\paragraph#1{\noindent \textbf{#1}}

\numberwithin{equation}{section}

\def\dist{\mathop{\rm dist}\nolimits}

\def\<{\langle}
\def\>{\rangle}

\def \ba {\begin{array}}
	\def \ea {\end{array}}

\newcommand{\be}{\begin{equation}}
	\newcommand{\ee}{\end{equation}}

\newcommand{\bea}{\begin{eqnarray}}
	\newcommand{\eea}{\end{eqnarray}}
\def\TH(#1){\label{#1}}\def\thv(#1){\ref{#1}}
\def\Eq(#1){\label{#1}}\def\eqv(#1){(\ref{#1})}
\def\cov{\hbox{\rm Cov}}
\def\var{\hbox{\rm Var}}

\def \1{\mathbbm{1}}

\def\wh{\widehat}

\newcommand{\bbE}{\mathbb{E}}

\newcommand{\bbP}{\mathbb{P}}

\newcommand{\bbZ}{\mathbb{Z}}
\newcommand{\bbR}{\mathbb{R}}

\newcommand{\cA}{\mathcal A}

\newcommand{\cC}{\mathcal C}

\newcommand{\cQ}{\mathcal Q}

\newcommand{\diam}{\mathop {\rm diam}}

\title{Second Order Asymptotics for the Hard Wall Probability of the 2D Harmonic Crystal}

\author[1]{Maximilian Fels}
\author[2]{Oren Louidor}
\author[2]{Tianqi Wu} 

\affil[1]{University of Cologne, Germany}
\affil[2]{Technion, Israel}\date{}

\begin{document}

\maketitle

\begin{abstract}
We estimate the probability that the discrete Gaussian free field on a planar domain with Dirichlet boundary conditions stays positive in the bulk. 
Improving upon the result by Bolthausen, Deuschel and Giacomin from 2001~\cite{bolthausen2001entropic}, we derive the order of the subleading term of this probability when a sequence of discretized scale-ups of given domain and compactly included smooth bulk are considered. A main ingredient in the proof is the double exponential decay of the right tail of the centered minimum of the field in the bulk, conditioned on a certain weighted average of its values to be zero.
\end{abstract}

\section{Introduction}
\subsection{Setup and Results}

Consider non-empty, bounded, simply connected, open sets $V \subset W \subset \mathbb{R}^2$ with $V$ compactly contained in $W$. Define scaled-up lattice versions $W_N := N W \cap \mathbb{Z}^2, V_N := N V \cap \mathbb{Z}^2$ for $N \in \mathbb{N}$. Let $h := h^{W_N} = (h^{W_N}(x) :\: x \in \bbZ^2)$ be the discrete Gaussian free field (DGFF) on $W_N$ with zero boundary condition, with law given by
\begin{equation}
\label{e:1.1}
	\bbP(dh) \propto e^{-\frac{1}{2} \langle h, Lh\rangle} \displaystyle\prod_{x\in W_N} dh(x)  \displaystyle\prod_{x\notin W_N} \delta_0(dh(x))
\end{equation}
where $L$ denotes the positive definite Laplacian $Lh(x) := \frac{1}{4}\sum_{y\sim x} (h(x) - h(y))$, with $\langle\cdot, \cdot\rangle$ being the $l^2$ inner product and $y\sim x$ if $y$ is a neighbor of $x$. It is known~\cite{ZeitouniBrahmson} that the minimum of $h^{W_N}$ is tight around $-m_N$,  where
\begin{equation}\label{eq:center_max}
	m_N := \bbE \max_{x\in W_N} h = 2\sqrt{g} (\log N  - \tfrac{3}{8} \log \log N) + O(1), \quad g: = \frac{2}{\pi}.
\end{equation}

We are interested in the probability of the ``hard wall'' event
\begin{equation}
	\Omega_{V_N}^+ := \left\{\min_{x\in V_N} h \geq 0\right\},
\end{equation}
which is a special case of the ``right tail'' event for the minimum (with $u = m_N$) 
\begin{equation}
\Omega_{V_N}(u) := \left\{\min_{x\in V_N} h \geq -m_N + u\right\}.
\end{equation}
Let us first introduce the (discrete) relative \emph{capacity} of $V_N$ with respect to $W_N$
\begin{equation}\label{def_capacity}
\operatorname{Cap}^{W_N}(V_N) := \inf_{h\in \mathbb{R}^{\mathbb{Z}^2}} \left\{ \frac{1}{2} \langle h, Lh\rangle: h|_{V_N} \equiv 1, h|_{W_N^\complement} \equiv 0\right\} = \frac{1}{2} \langle \psi_N, L\psi_N\rangle,
\end{equation}
where the \emph{capacitor} $\psi_N$ is the unique function that is harmonic on $W_N \setminus V_N$, equals $1$ in $V_N$ and $0$ in $W_N^\complement$. In particular, this quantity stays $\Theta(1)$ as $N$ grows (the proof of this well-known fact is given in the beginning of Section \ref{s:proofs_preliminaries} for the sake of completion):
\begin{lemma}\label{capacity_bound} For all $N$ large enough, the discrete relative capacity $\operatorname{Cap}^{W_N}(V_N) \in [c, C]$ for some constants $c, C \in (0, \infty)$ only depending on $V, W$.
\end{lemma}
\noindent

Our main results are the following upper and lower bounds for $\bbP(\Omega_{V_N}(u))$. 
\begin{theorem}\label{main_result} 
Assume $\partial V$ is $C^2$. There exist absolute constants $C,c \in (0,\infty)$ such that for all $N \geq 1$ and $u_0(V, W) \leq u\leq e^{\sqrt{\log N}}$,
	\begin{equation}\label{eq:DGFF_main_result}
		-\log \mathbb{P}(\Omega_{V_N}(u) ) =  \operatorname{Cap}^{W_N}(V_N) (u -  \theta_u \log u)^2 \,,
	\end{equation}
	where $\theta_u \in [c, C]$.
\end{theorem}

Specializing to the hard-wall event, we get,
\begin{corollary}
\label{c:1}
For all $N \geq 1$, $u \geq  u_0(V, W)$. 
	\begin{equation}
		-\log \mathbb{P}(\Omega_{V_N}^+ ) = 
		\operatorname{Cap}^{W_N}(V_N) \big(4g (\log N)^2 - \theta_N \log N \log \log N \big) \,,
	\end{equation}
where $\theta_N \in [c, C]$ for absolute constants $3\sqrt{g}/2 < c < C < \infty$.
\end{corollary}
The is as an improvement over the leading order asymptotics for the probability of the hard-wall event, which were derived in the seminal paper of Bolthausen, Deuschel and Giacomin~\cite{bolthausen2001entropic} for the case of $W = [0,1]^2$
(see the discussion in the end of this section).

\subsection{Top-level proof}
We begin by noticing that in the \emph{Dirichlet inner product} associated with $L$, the field $h$ has \emph{identity covariance}, so orthogonal decomposition in this inner product yields \emph{independence}. In particular, orthogonally projecting the field $h$ along the capacitor $\psi_N$ in the Dirichlet inner product gives that
\begin{equation}\label{eq:decomposition_in_Dirichlet}
	h - \frac{\langle h, L \psi_N\rangle}{\langle\psi_N, L \psi_N \rangle} \psi_N =: h - \sigma Z_h \psi_N
\end{equation}
is independent from
\begin{equation}\label{def_Z_h}
	Z_h := \sigma \langle h, L \psi_N\rangle \sim \mathcal{N}(0, 1)
\end{equation}
and has the same law as $h|Z_h = 0$, where 
\begin{equation}
	\sigma^{-2} := \var(\langle h, L \psi_N \rangle) = \langle \psi_N, L \psi_N \rangle = 2 \operatorname{Cap}^{W_N}(V_N).
\end{equation}
Since $h - \sigma Z_h \psi_N = h - \sigma Z_h$ on $V_N$, we can thus write
\begin{align}
\label{e:1.11}
	\mathbb{P}(\Omega_{V_N}(u) ) &= \mathbb{P}(\min_{x\in V_N} (h - \sigma Z_h) \oplus \sigma Z_h \geq -m_N + u) \\
	&= \int  \mathbb{P}(\sigma Z_h \in u -dv) \times \mathbb{P}(\min_{x\in V_N} h \geq -m_N + v| Z_h = 0) \\
	&=  \frac{1}{\sqrt{2\pi \sigma^2} }\int e^{- \frac{(u-v)^2}{2\sigma^2} } \mathbb{P}\left(\Omega_{V_N}(v) | Z_h = 0\right) dv \label{eq:DGFF_main_decomposition}
\end{align}
\noindent Above we used $\oplus$ to denote the sum of two independent quantities.

It remains to bound the tail probability in \eqref{eq:DGFF_main_decomposition}. To this end, we derive the following upper and lower bounds, which can be seen as the additional contribution of this work.
\begin{theorem}\label{main_tail_bound} Let $\beta \leq \sqrt{2\pi}/20$ and $N\in \mathbb{N}$. There exists $C = C(V, W) \in (0, \infty)$ such that for all $0\leq u\leq (\log N)^{2/3}$, 
	\begin{equation}\label{eq:main_tail_bound_upper}
		\bbP\left(\min_{x\in V_N} h(x) \geq -m_N + u \Big| \langle h, L \psi_N \rangle = 0\right) \leq 4 e^{- e^{\beta u}/C}. 
	\end{equation}
	For any $c > \sqrt{2\pi}$, there exists $C = C(V, c) \in (0, \infty)$ such that for all $1\leq u\leq \sqrt{\log N}$,
	\begin{equation}\label{eq:main_tail_bound_lower}
		\bbP\left(\min_{x\in V_N} h(x) \geq -m_N + u \Big| \langle h, L \psi_N \rangle = 0\right) \geq e^{-C e^{c u}}. 
	\end{equation}
\end{theorem}
\noindent
The proof of Theorem~\ref{main_tail_bound} constitutes the main effort in this work and occupies the remaining sections of this manuscript.

Let us now see how Theorem~\ref{main_tail_bound} can be used to derive the desired upper and lower bounds for the hard wall probability.
\begin{proof}[Proof of Theorem \ref{main_result}] Let $v_{\max} := (\log N)^{2/3}$. Plug the upper bound \eqref{eq:main_tail_bound_upper} into \eqref{eq:DGFF_main_decomposition}, this is at most
	\begin{align}
		& \frac{1}{\sqrt{2\pi \sigma^2} } \int_0^\infty 4 e^{\sigma^{-2} uv-\frac{v^2}{2\sigma^2}-e^{\beta (v \wedge v_{\max})}/C} \\ 
		&\lesssim \int_0^{v_{\max}} e^{\sigma^{-2} uv -e^{\beta v}/C} dv + e^{ \frac{u^2}{2\sigma^{2}}-e^{\beta v_{\max}}/C} \int_{v_{\max}}^{\infty} e^{-\frac{(v-u)^2}{2\sigma^{2}}} dv.
	\end{align}
	The integrand of the first term is maximized at $v^* = \beta^{-1} \log (\beta^{-1} C \sigma^{-2} u)$ and has exponential tail to its left and double exponential tail to its right, so the first integral is at most of order 
	\begin{equation}
	e^{\sigma^{-2} uv^*} \leq e^{\sigma^{-2} \beta^{-1} u \log u + C' u}.
	\end{equation}
	On the other hand, the second term is at most $e^{ \frac{u^2}{2\sigma^{2}}-e^{\beta (\log N)^{2/3}}/C}$, which is $o(1)$ for $u \leq e^{\beta/3\cdot (\log N)^{2/3}}$. This establishes the upper bound. For the lower bound, plug the lower bound \eqref{eq:main_tail_bound_lower} into \eqref{eq:DGFF_main_decomposition}, this is at least
	\begin{align}
		\frac{1}{\sqrt{2\pi \sigma^2}} \int_{1}^{\sqrt{\log N}} e^{\sigma^{-2} uv-\frac{v^2}{2\sigma^2}} \cdot e^{-Ce^{c v}} dv \geq e^{\sigma^{-2} u \cdot c^{-1} \log u - C'u}
	\end{align}
	where we restricted the integral to a window of length $1$ around $v =  c^{-1} \log u$. 
\end{proof}

\subsection{Discussion}
The DGFF is a canonical example of a Ginzburg-Landau (GL) Field, which is used, among other things, to model the height of an interface between two phases in a thermodynamical system. In this context, the event $\Omega_{V_N}^+$ represents an impenetrable ``wall' which is placed at height zero above the bulk of the
underlying domain of the system and thus prevents the interface from reaching below it. A well known consequence of the presence of such obstacle, which can already be seen in one dimension (the random walk case), is the repulsion of the field away from the wall, in order to make room for its typical entropic fluctuations. 

The arguably most ``physical'' case, which is also the one least mathematically complete, is when the underlying domain is of dimension $2$. This setup was studied in the mathematical literature in~\cite{bolthausen2001entropic, Daviaud} for the case of the DGFF and in~\cite{deuschel2000entropic} for more general GL Fields. In order to derive leading order asymptotics for $\bbP(\Omega_{V_N}^+)$, the authors in~\cite{bolthausen2001entropic} show that the typical height of the unconstrained minimum of the field on $V_N$ is $2\sqrt{g} \log N$ on first order. They then argue that for the field $h^{W_N}$ to be positive on $V_N$, it must ``lift'' on average to that level on $V_N$ so that ``typical'' minima are now positive. Computing the probabilistic ``cost'' of this lift gives the first order in~\eqref{eq:DGFF_main_result} by minimizing the exponent in~\eqref{e:1.1}. 

In this work we improve upon this leading order asymptotics by showing that the second order term of $\bbP(\Omega_{V_N}^+)$ is a negative bounded multiple of $(\log N)\log \log N$. We take advantage of the substantial progress made since~\cite{bolthausen2001entropic} in the study of the minimum of the field, which includes the derivation of its typical height up to an $O(1)$ additive term~\cite{ZeitouniBrahmson} and its weak convergence after centering~\cite{BDingZ}. 

The key idea here is to decompose the field as an independent sum of its orthogonal projection (w.r.t. the Dirichlet inner product) onto the (subspace spanned by the) capacitor $\psi_n$ and the remainder orthogonal projection on the complement space. The former, $\sigma Z_h \psi_N$, is a one-dimensional object (i.e. a one dimensional Gaussian-Hilbert space), which is a constant field on $V_N$ with height distributed as a centered Gaussian with variance $\sigma^2 = (2 \operatorname{Cap}^{W_N}(V_N))^{-1}$. Having this component reach level $u(1+o(1))$ on $V_N$ gives the leading order term for $\bbP(\Omega_{V_N}(u))$.

The second component in the orthogonal decomposition of the field has the law of $h$ conditioned on $\langle h, L \psi_N \rangle = 0$. Since the capacitor $\psi_N$ is harmonic in $W_N$ except on $\partial^{\text in} V_N$ - the inner boundary of $V_N$, it can be easily shown that the last inner product is, up to normalization by (twice) the capacity, a weighted average of $h$ over $\partial^{\text in} V_N$. The weights are given by the normalized escape probabilities of a random walk from the inner boundary of $V_N$ to $W^\complement_N$ and are $\Theta(1/N)$ whenever the boundary of $V$ is smooth. Thus $h|\langle h, L \psi_N \rangle = 0$ has roughly the law of a DGFF on $W_N$ conditioned to have average zero on $\partial^{\text in} V_N$.

As the first component in the decomposition is simply Gaussian, controlling the right tail of the minimum amounts to controlling the right tail of $h|\langle h, L \psi_N \rangle = 0$. This is the focus of Theorem~\ref{main_tail_bound} where it is shown that the decay of this tail is double exponential. This form of decay is the same as that of the right tail of the minimum of a DGFF on $V_N$, namely when the values of $h^{W_N}$ are conditioned to be exactly zero on $\partial^{\text in} V_N$. The double exponential decay of the right tail of the unconstrained minimum of the DGFF was shown by Ding in~\cite{Ding2013}, and we   indeed use this result as part of the proof of the upper bound in Theorem~\ref{main_tail_bound}.

As the integral in~\eqref{eq:DGFF_main_decomposition} is maximized at $v^* = \Theta(\log u)$, the proof of Theorem~\ref{main_result} can easily be used to show that conditional on $\Omega_{V_N}(u)$, the projection of $h$ onto the capacitor $\sigma Z_h \psi_N$ is $u - \Theta(\log u)$ on $V_N$, with complementary probability decaying at lest exponentially in $u$. Nevertheless, there does not seem to be an easy way to turn this into a more explicit statement about the lift of the field under the conditioning, involving, e.g. the conditional mean of the field at a vertex of $V_N$ or the empirical average of the field on $V_N$. 

In particular, the fact that $\sigma Z_h \psi_N$ is $2\sqrt{g}\log N - \Theta(\log \log N) $ on $V_N$ under $\Omega_{V_N}^+$ with overwhelmingly high probability, seems to indicate that the conditional mean of $h$ under the same event is $2\sqrt{g}\log N - \Theta(\log \log N)\ll 2\sqrt{g}\log N$ as well. This would have been an improvement over the known leading order asymptotics of this conditional mean, which were indeed shown in~\cite{bolthausen2001entropic} to be $2\sqrt{g} \log N(1+o(1))$. However, our methods fall short of doing as much.

Lastly, we remark that in the hierarchal version of this problem, namely for the DGFF on a regular tree, sharp asymptotics (up-to $o(n)$, where $n$ should be compared to $\log N$ in this work) for the probability of the hard wall event was shown recently in a sequence of works by two of the present authors together with Hartung~\cite{BRWHW1, BRWHW2}. These works also include a derivation of sharp asymptotics for the repulsion level under the positivity constraint, as well as a full asymptotic description of the law of the conditional field. In particular, the typical lift is smaller than the typical height of the minimum by a diverging quantity ($C\log n + O(1)$) and the law of the recentered conditional field is not asymptotically the same as that of the unconditional field (but the laws of the gradients do match locally in the limit). We conjecture that this is the same in the present case of the plane.

\subsection*{Organization of the paper}
The rest of the paper is organized as follows. In Section~\ref{s:preliminaries} we introduce additional notation and state some auxiliary lemmas to be used in the remainder of this work. Section~\ref{s:u} is devoted to proving the upper bound in Theorem~\ref{main_tail_bound}. The lower bound is treated in Section~\ref{s:l}. Lastly, proofs of the auxiliary lemmas are given in~\ref{s:a}. 

\section{Preliminaries}\label{s:preliminaries}
\subsection{Notation}
Let us collect the general notation used in this work. For general Gaussian processes, we denote the law by $\mathbb{P}$ and the expectation by $\mathbb{E}$. On the other hand $\mathsf{P}^x$ (resp.\ $\mathsf{P}^\mu$) and $\mathsf{E}^x$ (resp.\ $\mathsf{E}^\mu$) will be used for the law of a discrete time random walk $(S_t)_{t\in \mathbb{N}}$ on $\mathbb{Z}^2$ starting from $x\in \mathbb{Z}^2$ (resp.\ initial distribution $\mu$) and the associated expectation.
We shall $\oplus$ to indicate the sum of independent random variables. 

Given a subset $A \subset \mathbb{Z}^2$, we let $\tau_A := \inf\{t \geq 0 : S_t \in A\}$ be the first hitting time of $A$, and $\tau_A^+ := \inf\{t \geq 1 : S_t \in A\}$ be the first return time. We denote the harmonic measure from $x$ in $A$ by $\Pi_A(x, \cdot) = \mathsf{P}^x(S_{\tau_{A^\complement}} \in \cdot)$.
The Green Function on $W \subset \mathbb{Z}^2$, is given by
\begin{equation}
	G_W(x, y) := \mathsf{E}^x \left[\sum_{t=0}^{\tau_{W^\complement}-1} \mathbf{1}_{\{S_t = y\}} \right], \quad x, y \in \mathbb{Z}^2.
\end{equation}
For a set $A \subset \mathbb{Z}^2$, we extend this definition to $G_W(x, A)$ by replacing the indicator $\mathbf{1}_{\{S_t = y\}}$ with $\mathbf{1}_{\{S_t \in A\}}$ in the formula above.
The process $h^W = (h^W(x) :\: x \in \bbZ^2)$ will always denote the Discrete Gaussian Free Field (DGFF) on $W$ with zero boundary condition outside (i.e.~$h(x) \equiv 0$ for $x\notin W$), which makes it a Gaussian process with mean zero and covariance given by $G_W$.

The diameter of a set, distance from a point to a set and distance between two sets will be defined in the usual sense, using the Euclidean Norm, which we denote by $\diam$ and $\dist$.
For a discrete set $\Lambda \subset \mathbb{Z}^2$, we write $|\Lambda|$ (or $\# \Lambda$) for its cardinality. The inner/outer boundaries and interior of $\Lambda$ are given by
\begin{align}
	\partial^{\text{in}} \Lambda &:= \{ x \in \Lambda : \exists y \in \mathbb{Z}^2 \setminus \Lambda, |x-y|_1 = 1 \}, \\
	\partial^{\text{out}} \Lambda &:= \{ y \in \mathbb{Z}^2 \setminus \Lambda : \exists x \in \Lambda, |x-y|_1 = 1 \}, \\
	\Lambda^- &:= \Lambda \setminus \partial^{\text{in}} \Lambda. \label{eq:kappa}
\end{align}
The discrete ball (resp.~box) of radius (resp.~side-length) $r$ centered at $x$ will be denoted by
\begin{equation}\label{notation_discrete_ball}
	\cC_r(x) := \{y\in \mathbb{Z}^2: \|y-x\|_2 < r\}
\end{equation}
and
\begin{equation}\label{notation_discrete_box}
			Q_r(x) := \{y\in \mathbb{Z}^2: \|y-x\|_\infty \leq r/2\}.
\end{equation}

For an open set $U \subset \mathbb{R}^2$ in the continuum, we use $\partial U$ to denote its topological boundary. We define the $\delta$-interior of $U$ as
	\begin{equation}\label{def_delta_interior}
			U^{\delta} := \{x\in U: \dist(x, \partial U) > \delta\}.
		\end{equation}
If $\partial U$ is a compact, $C^2$-regular planar curve, the signed curvature $\kappa(x)$ is well-defined, and we let $\kappa_{\max}:= \max_{x\in \partial U} |\kappa(x)|$. By \cite[Lemma 14.16]{GilbargTrudinger2001} and the implicit function theorem, observing that  $\partial U^\delta$ can be written as a level set of the distance function with respect to $\partial U$, the set $\partial U^\delta$ remains $C^2$-regular provided $\delta < \kappa_{\max}^{-1}$, and its curvature bounded by $\kappa_{U^{\delta}}=\frac{\kappa_{\max} }{1-\delta \kappa_{\max} }$.

Finally, $a\lesssim b$ indicates that $a\leq C b$ for some absolute constant $C>0$. We write $a \asymp b$, or equivalently $\Theta(a)=b$, if both $a\lesssim b$ and $b\lesssim a$. The notation $a\ll b$ implies that the ratio $a/b$ is sufficiently small.

\subsection{Auxiliary lemmas}
Now we state some auxiliary lemmas used in our work. The proofs of these  standard and/or straight-forward results are deferred to Section \ref{s:proofs_preliminaries}. 

\begin{lemma}\label{remove_conditioning_trick} If $\vec X, \vec Y$ are jointly Gaussian vectors, then for any $\vec x, \vec y$
	\begin{equation}
		\bbP\left(\min_{i} (X_i - x_i) \leq 0|\vec Y = \vec y\right) \leq \frac{	\bbP\left(\displaystyle\min_{i} (X_i - x_i) \leq 0 \right)}{\displaystyle\min_{i} \bbP\left(\bbE[X_i|\vec Y] \leq \bbE[X_i|\vec Y = \vec y]\right)}.
	\end{equation}
	In particular, if $\vec Y$ is centered, we have
	\begin{equation}\label{conditioning_zero_trick}
		\bbP\left(\displaystyle\min_{i} (X_i-x_i) \leq 0\right) \geq \frac{1}{2}\bbP\left(\min_{i} (X_i-x_i) \leq 0|\vec Y = \vec 0\right) .
	\end{equation}
\end{lemma}

\begin{lemma}\label{conditional_expectation_of_Gaussian} If $X, \vec Y$ are jointly Gaussian with $\vec Y$ centered, and $\Sigma$ is the covariance matrix of $\vec{Y}$, then for $\vec y \in \operatorname{Range}(\Sigma)$,
	\begin{equation}
		\bbP\left(\bbE[X|\vec Y] \leq \bbE[X|\vec Y = \vec y]\right) \geq \Phi_{\mathcal{N}(0, 1)}\left(-\sqrt{\vec y \cdot \Sigma^\dagger \vec y}\right).
	\end{equation}
where $\Sigma^\dagger$ denotes the Moore-Penrose pseudo-inverse of $\Sigma$, and $\Phi_{\mathcal{N}(0, 1)}$ denotes the CDF of the standard normal distribution.  
\end{lemma}

\begin{lemma}\label{lemma:magnitude_Delta}
	Let $h$ be a DGFF on $W \subset \mathbb{Z}^2$ with zero boundary condition. Given $V \subset W$, let $\phi^{W, V}$ be the harmonic extension of $h|_{V^\complement}$ on $V$. There exists an \emph{absolute} constant $C>0$ such that for any $x \in V \subset W \subset \bbZ^2$,
	\begin{equation}
		g \log \left(\frac{\dist(x, W^\complement)}{\diam(V)}\right) - C \leq \var(\phi^{W, V}(x)) \leq g \log \left(\frac{\diam (W)}{\dist(x, V^\complement)}\right) + C.
	\end{equation}
\end{lemma}

\begin{lemma}\label{lemma:diff_Delta_squared_modified}
	Let $h$ be a DGFF on $W \subset \mathbb{Z}^2$ with zero boundary condition. Given $V \subset W$, let $\phi^{W, V}$ be the harmonic extension of $h|_{V^\complement}$ on $V$. There exists an \emph{absolute} constant $C>0$ such that for any $x,y\in U \subset V \subset W \subset \bbZ^2$,
	\begin{equation}
		\mathbb{E}\left[\left( \phi^{W, V}(x)-\phi^{W, V}(y)\right)^2\right]\leq C \frac{\|x-y\|_2}{\dist(U,V^\complement)} \leq C \frac{\diam(U)}{\dist(U,V^\complement)}.
	\end{equation}
\end{lemma}

\begin{lemma}\label{lemma:dudley_modified} Let $\Delta$ be a mean zero Gaussian field on $U \subset \bbZ^2$ satisfying 
	\begin{equation}\label{eq:dudley_assumption_2}
		\bbE[| \Delta (x)-\Delta(y)|^2] \leq L \|x-y\|_2
	\end{equation}
	for all $x, y$ in $U$. Then there exists an \emph{absolute} constant $K > 0$ such that
	\begin{equation}\mathbb{E}\left[\sup_{x,y\in U} | \Delta (x)-\Delta(y)|\right]\leq K c_U^{-\tfrac 14}\sqrt{L \diam(U)}.
	\end{equation}
	where $c_U := \inf \left\{\frac{| \cC_r(x) \cap U|}{r^2}: x\in U, 0 < r\leq \diam U \right\}$. (Here $\cC_r(x)$ is the discrete ball defined in \eqref{notation_discrete_ball}.)
\end{lemma}

\begin{lemma}\label{lemma:escape_probability_general}
	Let $U \subset \mathbb{Z}^2$ be a discrete domain. Suppose the outer boundary $\partial^{\text{out}} U$ is partitioned into two disjoint sets: a trap $E$ and an escape region $D$. We assume there exists an open set $\Gamma \subset \mathbb{R}^2$ such that
	\begin{equation}\label{eq:Gamma_assumption_1}
		\sup_{x\in E} \dist (x, \partial \Gamma) \leq M, \quad \sup_{x\in U} \dist (x, \Gamma^\complement) \leq M
	\end{equation}
	for some $1 \leq M < \infty$, and $\partial \Gamma$ is a $C^2$ curve with curvature absolutely bounded by
	\begin{equation}
		\kappa_{\max}\leq \max(2M, \diam(U)/R)^{-1}
	\end{equation}\label{eq:Gamma_assumption_curvature}
	for some $0 < R < \infty$. Then there exists a constant $C = C(M, R) > 0$ such that for all $x\in U$,
	\begin{equation}
	\mathsf{P}^{x}\left(\tau_{D} < \tau_{E}\right) \le C \frac{\dist(x, E)}{\dist(x, D)}.
	\end{equation}
\end{lemma}

\section{Proof of the upper bound in Theorem \ref{main_tail_bound}}
\label{s:u}
We begin with a ``binding field'' decomposition for $h|Z_h=0$ on $V_N$. Let $\phi := \phi^{W_N, V_N^-}$ be the harmonic extension of $h|_{(V_N^-)^\complement}$ on $V_N^-:= V_N\setminus \partial^{\text{in}} V_N$, then we have the Gibbs-Markov decomposition
\begin{equation}\label{eq:binding_field_decomposition}
	h = \phi^{W_N, V_N^-} \oplus h^{V_N^-} = \phi \oplus h',
\end{equation}
where $h' := h^{V_N^-}$ is independent from $\phi$ and has the law of DGFF on $V_N^-$ with zero boundary condition. Since $Z_h$ defined in \eqref{def_Z_h} is a linear combination of $h|_{\partial^{\text{in}}  V_N}$ and $\partial^{\text{in}} V_N \subset (V_N^-)^\complement$, $Z_h$ is a function of $\phi$, so conditioning on $Z_h = 0$ yields that
\begin{equation}
	(h|Z_h = 0) = \Delta \oplus h', \qquad \Delta := (\phi|Z_h = 0).
\end{equation}

Our strategy is to uniformly control the $\Delta$ field at exponentially many small boxes near the boundary of $V_N$ and then control the $h'$ field in these boxes by adapting the argument of \cite[Section 2.4]{Ding2013}. Let $c > 0$ be a parameter to be specified later. For $1 \ll u\leq (\log N)^{2/3}$, let us introduce the following scales
\begin{itemize}
	\item $r := N e^{-cu}$. 
	\item $r' := Ne^{-cu (1+\eta)}$ for $\eta\geq 0$. 
	\item $l := N e^{-cu(1-\delta)}$ for $\delta \in (0, 1)$. 
\end{itemize}

Let $V_{N, r} := N V^{r/N} \cap \mathbb{Z}^2$, where $V^{r/N}$ is the $r/N$-interior of $V$ defined in \eqref{def_delta_interior}. As in \eqref{notation_discrete_box}, we denote by $Q_{r'}(x)$ the $r' \times r'$ box centered at $x$. 

\begin{proposition}\label{SS-dilution} For $\eta \in (0, 1)$ and $\delta \in (0, 1/2)$, there exists a constant $C < \infty$ such that with complementary probability at most $e^{-e^{cu\eta}/C}$, there exists $m \geq e^{cu(1/2-\delta)}/C$ points $x_1, x_2, \cdots, x_m$ in $\partial^{\text{in}} V_{N,r}$, separated by distance at least $l$ from each other, such that $\Delta(x) \leq 3$ for $x \in \bigcup_{i=1}^m Q_{r'}(x_i)$. 
\end{proposition}

This main Proposition will be proved in the next Subsections. 

\begin{lemma}\label{lemma:Ding_argument} Given points $x_1, x_2, \cdots, x_m \in \partial^{\text{in}} V_{N,r}$ separated by distance at least $l$ from each other, let $\cQ_{r'} = \{Q_{r'}(x_i): 1\leq i\leq m\}$. There exists absolute constant $C > 0$ such that
	\begin{equation}\label{event_min_h'_boxes}
		\bbP\left(\min_{x\in \cup \cQ_{r'}} h'(x)  \geq -m_{r'} + C \right) \leq  e^{-(l/r)/C} + 2 e^{-m/C}.
	\end{equation}
\end{lemma}

The argument for this lemma is basically \cite[Section 2.4]{Ding2013} adapted to domain with more general shapes. For sake of completeness, we give a proof of this fact in Subsection \ref{s:proof_Ding_argument}. Let us now see how the two preceding results lead to the upper bound in Theorem \ref{main_tail_bound}.

\begin{proof}[Proof of Theorem \ref{main_tail_bound}, upper bound] Fix a realization of $\Delta$ and points $x_1, x_2, \cdots, x_m$ in $\partial^{\text{in}} V_{N,r}$ as in Lemma \ref{SS-dilution}. Denote $\cQ_{r'} = \{Q_{r'}(x_i): 1\leq i\leq m\}$. By independence of $\Delta$ and $h'$, we can apply Lemma \ref{lemma:Ding_argument} to $h'$ \emph{conditioned on this realization of} $\Delta$ and ensure that
	\begin{equation}
		\min_{x\in \cup \cQ_{r'}} h'(x)  \leq -m_{r'} + C \overset{\eqref{eq:center_max}}{\leq} -m_N + 2\sqrt{2/\pi}\cdot cu(1+\eta) + C'
	\end{equation}
	with complementary probability at most 
	\begin{equation}
		e^{-(l/r)/C} + 2 e^{-m/C} \leq e^{- e^{cu\delta}/C} + 2 e^{-e^{cu(1/2-\delta)}/C}.
	\end{equation}
	Set $\delta = 1/4$, and $c (1+\eta) = 1/(2\sqrt{2/\pi})$, we conclude that
	\begin{equation}\label{event_min_h'_boxes_applied}
		\bbP\left(\min_{x\in \cup \cQ_{r'}} h'(x)  \geq -m_N + u \right) \leq  3 e^{-e^{c u/4}/C}. 
	\end{equation}
	Set $\eta = 1/4$, then together with Lemma \ref{SS-dilution}, we obtain overall
	\begin{equation}
		\bbP\left(\min_{x\in V_N} \Delta(x) \oplus h'(x)  \geq -m_N + u\right) \leq  4 e^{-e^{cu/4}/C}
	\end{equation}
	where $c = 4/5 \times 1/(2\sqrt{2/\pi}) = 4\cdot \sqrt{2\pi}/20$.
\end{proof}

\subsection{Proof of Proposition \ref{SS-dilution}}

The key to prove Proposition \ref{SS-dilution} is to show that an ``average'' of the $\Delta$ field on $\partial^{\text{in}} V_{N,r}$ is small. Precisely, 

\begin{lemma}\label{variance_average_Delta} There exists a probability measure $\gamma$ on $\partial^{\text{in}} V_{N,r}$ satisfying
	\begin{equation}\label{eq:assumption_for_gamma}
		\forall x\in \mathbb{Z}^2: \gamma(Q_l(x)) \lesssim l/N
	\end{equation}
	such that the random variable 
	\begin{equation}
		\overline \Delta := \sum_{x\in \partial^{\text{in}} V_{N,r}} \gamma(x) \Delta(x)
	\end{equation} 
	is Gaussian with $\var(\overline \Delta) \lesssim e^{-cu}$. 
\end{lemma}

We postpone the proof of this key Lemma to Section \ref{s:variance_average_Delta}. In addition, we need to control the \emph{oscillation} of $\Delta$ in boxes $Q_{r'}(x)$ for any given $x\in V_{N, r}$.

\begin{lemma}\label{oscillation_Delta} There exists constant $C < \infty$ such that for all $x \in V_{N, r}$ and all $t \geq 1$,
	\begin{equation}\label{eq:Borell_TIS_applied}
		\mathbb{P}\left(\sup_{y\in Q_{r'}(x)} | \Delta (y)-\Delta(x)| > t \right)\leq C \exp\left(-t^2 e^{cu\eta}/C\right).
	\end{equation}
\end{lemma}

\begin{proof}[Proof of Lemma \ref{oscillation_Delta}] Given $x \in V_{N, r}$, set $\mu :=\mathbb{E}\left[\sup_{y\in Q_{r'}(x)} | \Delta (x)-\Delta(y)|\right]$ and $\sigma^2:= \sup_{y\in Q_{r'}(x)} \mathrm{Var}[\Delta(x)-\Delta(y)]$. By Borell-TIS,
	\begin{equation}\label{eq:Borell_TIS_for_boxes}
		\mathbb{P}\left( \left|\sup_{y\in Q_{r'}(x)} | \Delta (x)-\Delta(y)| -\mu\right| > t \right)\leq 2 \exp\left(-\frac{t^2}{2\sigma^2}\right).
	\end{equation}
	
	Applying Lemma \ref{lemma:diff_Delta_squared_modified} with $W = W_N, V= V_N^-$, and $U = Q_{r'}(x)$, together with the fact that conditioning on $Z_h = 0$ \emph{only reduces variance} of $\phi^{W_N, V_N^-}$, we get $\sigma^2 \lesssim r'/r = \frac{Ne^{-cu(1+\eta)}}{Ne^{-cu}} = e^{-cu\eta}$. In this case, \eqref{eq:dudley_assumption_2} holds with $L \asymp 1/r$ and $c_U \gtrsim 1$, so by Lemma \ref{lemma:dudley_modified} $\mu \lesssim \sqrt{\frac{r'}{r}} = e^{-cu\eta/2} \leq 1$. Plugging the bounds for $\mu$ and $\sigma^2$ into \eqref{eq:Borell_TIS_for_boxes}, we obtain that for some absolute constant $C < \infty$, it holds for all $t \geq 1$ that
	\begin{equation}\label{eq:Borell_TIS_for_boxes_applied}
		\mathbb{P}\left(\sup_{y\in Q_{r'}(x)} | \Delta (x)-\Delta(y)| > t \right)\leq C \exp\left(-t^2 e^{cu\eta}/C\right).
	\end{equation}
\end{proof}

We conclude the present section by proving Lemma \ref{SS-dilution} using these two ingredients.

\begin{proof}[Proof of Lemma \ref{SS-dilution}]  
By Lemma \ref{variance_average_Delta}, with complementary probability at most $e^{-e^{cu}/C}$
	\begin{equation}\label{event_average_Delta_small}
	\sum_{x\in \partial^{\text{in}} V_{N,r}} \gamma(x) \Delta(x) \leq 1.
	\end{equation}
	On this event, we have 
	\begin{align}
		\sum_{x\in \partial^{\text{in}} V_{N,r}: \Delta(x)-2\leq 0} \gamma(x) (\Delta(x)-2) &\leq -1 \\
		\left(\min_{x\in \partial^{\text{in}} V_{N,r}} \Delta(x)-2 \right) \times 
		\gamma\Big(x\in \partial^{\text{in}} V_{N,r}: \Delta(x)-2 \leq 0\Big)&\leq - 1.
	\end{align}
	Let $S = \{x\in \partial^{\text{in}} V_{N,r}: \Delta(x) \leq 2\}$. The inequality above implies either
	\begin{enumerate}[label=(\textbf{\roman*})]
		\item there exists some $x\in \partial^{\text{in}} V_{N,r}$ with $\Delta(x)-2 \leq -e^{cu/2}$, or
		\item $\gamma(S) \geq e^{-cu/2}$.
	\end{enumerate}
By a Union Bound \emph{over boxes} of side length $r = Ne^{-cu}$, the event \textbf{(i)} occurs with probability at most 
\begin{equation}
 (C e^{cu})^2 \times 2e^{-e^{cu}/(Cu)} \leq e^{-e^{cu}/(C'u)}
\end{equation}
since, by Lemma \ref{lemma:magnitude_Delta}, $\var(\Delta(x)) = O(u)$ for $x \in V_{N,r}$ and, by Lemma \ref{oscillation_Delta} with $\eta = 0$, the \emph{oscillation} of $\Delta$ is $O(1)$ sub-Gaussian in a $r \times r$ box centered at $x$. Together with \eqref{event_average_Delta_small}, this means the event \textbf{(ii)} occurs with complementary probability at most $e^{-e^{cu}/(C'u)}$. 

Now, let $\cQ_l = \{Q_l(y): y \in (l\mathbb{Z})^2, Q_l(y) \cap S \neq \emptyset\}$ be the collection of $l \times l$ boxes intersecting $S$. Then $S \subset \bigcup \cQ_l$. In the event \textbf{(ii)}, we have
\begin{equation}
|\cQ_l| \geq \frac{\sum_{Q \in \cQ_l} \gamma(Q)}{\max_{Q \in \cQ_l} \gamma(Q)} \gtrsim \frac{\gamma(S)}{l/N} \geq \frac{e^{-cu/2}}{ e^{-cu(1-\delta)}} = e^{cu (1/2-\delta)}
\end{equation}
where we used \eqref{eq:assumption_for_gamma} to bound the denominator. Moreover, by identifying $\cQ_l \subset (l\mathbb{Z})^2$, we can partition it into $4$ subcollections by the quotient operation $(l\mathbb{Z})^2/(2l\mathbb{Z})^2 \simeq (\mathbb{Z}/2\mathbb{Z})^2$. Then one of these subcollections will have at least $1/4$ of the boxes in $\cQ_l$. This means we can find points $x_1, x_2, \cdots, x_m \in S$ at distance at least $l$ from each other, with $m \geq |\cQ_l|/4 \gtrsim e^{cu (1/2-\delta)}$. 

Finally, by Lemma \ref{oscillation_Delta} with $\eta \geq 1$ and Union Bound, we can also ensure the oscillation of $\Delta$ is bounded by $1$ in every $r'\times r'$ box intersecting $V_{N,r}$, with complementary probability at most
\begin{equation}\label{event_oscillation_Delta_bounded}
	(C e^{cu(1+\eta)})^2 \times e^{-e^{cu\eta}/C} \leq e^{-e^{cu}/C'}.
\end{equation}
On this event, $\Delta \leq 2 + 1 = 3$ on $\bigcup_{i=1}^m Q_{r'}(x_i)$.

\end{proof}

\subsection{Proof of Lemma \ref{variance_average_Delta}}\label{s:variance_average_Delta}
To prove this Lemma, we first rewrite the average $\overline \Delta$ in a different way. Recall $\Delta (x) = (\phi(x) | Z_h = 0)$. Define
\begin{align}
	\overline \phi &:=\sum_{x\in \partial^{\text{in}} V_{N,r}} \gamma(x) \phi(x), 
\end{align}
so $\overline \Delta = (\overline \phi|Z_h=0)$. By Gaussianity,
\begin{equation}
	\overline \Delta = \overline \phi - \bbE[\overline \phi|Z_h].
\end{equation}
Recall from \eqref{eq:decomposition_in_Dirichlet} that $\bbE[h(z)|Z_h] = \sigma Z_h$ in $V_N$. This implies that $\bbE[\phi(x)|Z_h] = \sigma Z_h$ in $V_N$, because the \emph{binding field} $\phi := \phi^{W_N, V_N^-}$ in $V_N$ is a \emph{harmonic average} of $h$ on $V_N\setminus V_N^- = \partial^{\text{in}} V_N$, i.e. for all $x \in V_N$,
\begin{equation}\label{eq:harmonicity_binding_field}
	\phi(x) = \mathsf{E}^x h(S_{\tau_{\partial^{\text{in}} V_N}}).
\end{equation}
Since $\overline \phi$ is further defined as a weighted average of $\phi$ in $V_N$, it follows that
\begin{equation}
	\bbE [\overline \phi|Z_h] = \sigma Z_h.
\end{equation}
Thus, $\overline \Delta = \overline \phi - \sigma Z_h$ and
\begin{align}\label{eq:var_overline_Delta_decomp}
	\var(\overline \Delta) = \var(\overline \phi) - \var(\sigma Z_h)
\end{align}
where $\var(\sigma Z_h) = \sigma^2 = (2\operatorname{Cap}^{W_N}{V_N})^{-1} = \Theta(1)$. Thus we want to choose $\gamma$ so that
\begin{equation}\label{eq:Lemma_3.3_goal}
    \var(\overline \phi) = (1+O(\tfrac{r}{N})) \var(\sigma Z_h).
\end{equation}

The following lemma gives a nice representation of $\sigma Z_h$ as a weighted average of $h$ on $\partial^{\text{in}} V_N$:

\begin{lemma}\label{lem:sigma_Z_representation} We have \begin{equation}
		\sigma Z_h = \sum_{z\in \partial^{\text{in}} V_N} \gamma^*(z) h(z)
	\end{equation}
where $\gamma^*$ is the probability measure on $\partial^{\text{in}} V_N$ defined by
	\begin{equation}\label{eq:gamma*_random_walk}
	\gamma^*(z) := \frac{\mathsf{P}^z(\tau_{W_N^\complement} < \tau^+_{V_N})}{2 \operatorname{Cap}^{W_N}(V_N)} = \mathsf{P}^{\operatorname{Unif}(\partial^{\text{out}} W_N)}(S_{\tau_{V_N} } = z | \tau_{V_N} < \tau^+_{W_N^\complement})
\end{equation}
where $\operatorname{Unif}(\partial^{\text{out}} W_N)$ is the uniform probability measure on $\partial^{\text{out}} W_N$.
\end{lemma}
\begin{proof}[Proof of Lemma \ref{lem:sigma_Z_representation}] From the defining equation \eqref{eq:decomposition_in_Dirichlet}, we have 
	\begin{equation}
		\sigma Z_h = \frac{\langle h, L \psi_N\rangle}{\langle\psi_N, L \psi_N \rangle} = \frac{ \displaystyle \sum_{z\in \partial^{\text{in}} V_N} \mathsf{P}^z(\tau_{W_N^\complement} < \tau^+_{V_N}) h(z)}{ \displaystyle  \sum_{z\in \partial^{\text{in}} V_N} \mathsf{P}^z(\tau_{W_N^\complement} < \tau^+_{V_N})}
	\end{equation}
	by interpreting the harmonic function $\psi_N(x)$ as the probability of a simple random walk started at $x$ reaching $V_N$ before exiting $W_N$. Notice the denominator is exactly $\sigma^{-2} = 2 \operatorname{Cap}^{W_N}(V_N)$ and normalizes the coefficients of $h$ in the numerator to a probability measure. Time reversal then gives the interpretation \eqref{eq:gamma*_random_walk} of $\gamma^*$.
\end{proof}

Motivated by this representation of $\sigma Z_h$, we will choose 
\begin{equation}\label{eq:gamma_random_walk}
	\gamma(x) := \frac{\mathsf{P}^x(\tau_{W_N^\complement} < \tau^+_{V_{N,r}})}{2 \operatorname{Cap}^{W_N}(V_{N,r})} = \mathsf{P}^{\operatorname{Unif}(\partial^{\text{out}} W_N)}(S_{\tau_{V_{N,r}} } = x | \tau_{V_{N,r}} < \tau^+_{W_N^\complement}).
\end{equation}

\begin{lemma}\label{lem:check_assumption_for_gamma} The measure $\gamma$ defined by \eqref{eq:gamma_random_walk}  satisfies \eqref{eq:assumption_for_gamma}.
\end{lemma} 
\begin{proof}[Proof of Lemma \ref{lem:check_assumption_for_gamma}]
	Recall that $V_{N,r} := N V^{r/N} \cap \mathbb{Z}^2$ where 
	\begin{equation}
		V^{r/N} := \{x\in V: \dist(x, \partial V) > r/N\}.
	\end{equation}
	As we work in the regime $u \gg 1$, $r/N = e^{-cu} \ll 1$, so $\operatorname{Cap}^{W_N}(V_{N,r}) \asymp 1$ by Lemma \ref{capacity_bound} applied to $V_{N,r}$. Thus,
	\begin{equation}\label{eq:gamma_rewrite}
		\gamma(x) \asymp \mathsf{P}^x(\tau_{W_N^\complement} < \tau^+_{V_{N,r}}) = \frac{1}{4} \sum_{y\sim x, y\in \partial^{\text{out}} V_{N,r}} \mathsf{P}^y(\tau_{W_N^\complement} < \tau_{V_{N,r}}).
	\end{equation}
	Since $\partial V$ is $C^2$, $\partial V^{r/N}$ is also $C^2$ for $r/N \ll 1$, and of uniformly bounded curvature as explained in the second paragraph following \eqref{eq:kappa}. Thus we can apply Lemma \ref{lemma:escape_probability_general} (with $U = W_N \setminus V_{N,r}$, $E = \partial^{\text{in}} V_{N,r}$, $D = \partial^{\text{out}} W_N$, $\Gamma = N V^{r/N}$) at each of the (at most four) neighbors of $x$ in $\partial^{\text{out}} V_{N,r}$, giving an upper bound of a constant times $1/N$ in \eqref{eq:gamma_rewrite}. Since $\partial V^{r/N}$ is $C^2$, it also follows $|\partial^{\text{in}} V_{N,r} \cap Q_l(x)| \lesssim l$. Thus, summing up we get $\gamma(Q_l(x)) \lesssim \tfrac{l}{N}$, as required in \eqref{eq:assumption_for_gamma}. 
\end{proof}

It remains to show \eqref{eq:Lemma_3.3_goal} for this choice of $\gamma$. We begin with a \emph{random walk representation} of $\overline \phi$. As before, $(S_t: t\in \mathbb{N})$ denotes a simple random walk. By harmonicity of the binding field $\phi$, we have
\begin{equation}\label{eq:overline_phi_weighted_average}
	\overline \phi = \mathsf{E}^\gamma \, \phi(S_0)  \overset{\eqref{eq:harmonicity_binding_field}}{=} \mathsf{E}^\gamma [h(S_{\tau_{\partial^{\text{in}} V_N}})].
\end{equation}
Consequently, we have that for $y\in \partial^{\text{in}} V_N$,
\begin{equation}\label{eq:overline_phi_as_Green_function}
	\bbE[h(y) \cdot \overline \phi] = \bbE[h(y) \mathsf{E}^\gamma\, h(S_{\tau_{\partial^{\text{in}}  V_N}})] = \mathsf{E}^\gamma \, G_{W_N}(S_{\tau_{\partial^{\text{in}}  V_N}}, y) = \mathsf{E}^\gamma\, G_{W_N}(S_0, y)
\end{equation}
by harmonicity of the Green function $G_{W_N}(\cdot, y)$ in $W_N\setminus\{y\}$. 
\begin{lemma}\label{lem:variance_of_overline_phi} Denote by $\gamma \Pi$ the law of $S_{\tau_{\partial^{\text{in}}  V_N}}$ under $\mathsf{P}^\gamma$. Then
	\begin{equation}
		\var(\overline \phi) =  \frac{\mathsf{P}^{\gamma \Pi}(\tau_{V_{N,r}} < \tau_{W_N})}{2\operatorname{Cap}^{W_N}(V_{N,r})}.
	\end{equation}
\end{lemma}

\begin{proof}[Proof of Lemma \ref{lem:variance_of_overline_phi}] Let $T$ be the \emph{last time} the simple random walk $(S_t: t\in \mathbb{N})$ visits $V_{N,r}$ before exiting $W_N$, and we define $T = \infty$ if this never happens. For any $y \in W_N$,
	\begin{align}
		\mathsf{E}^\gamma \, G_{W_N}(y, S_0) &= \sum_{x\in \partial^{\text{in}} V_{N,r}} G_{W_N}(y, x) \gamma(x) \\
		&= \sum_{x\in \partial^{\text{in}} V_{N,r}} \mathsf{E}^y \left[\sum_{n=0}^\infty 1_{\{S_n = x\}} 1_{ \{ \tau_{W_N^\complement} > n\}}\right]  \times \frac{ \mathsf{P}^x(\tau_{W_N^\complement} < \tau^+_{V_{N,r}})}{\sigma_{N,r}^{-2}} \\
		&= \sigma_{N,r}^{2} \sum_{x\in \partial^{\text{in}} V_{N,r}}  \sum_{n=0}^\infty  \mathsf{P}^y (S_n = x, \tau_{W_N^\complement} > n) \mathsf{P}^x(\tau_{W_N^\complement} < \tau^+_{V_{N,r}}) \\
		&= \sigma_{N,r}^{2} \sum_{x\in \partial^{\text{in}} V_{N,r}}  \sum_{n=0}^\infty  \mathsf{P}^y (T = n, S_T = x) \\
		& = \frac{\mathsf{P}^y(\tau_{V_{N,r}} < \tau_{W_N})}{2 \operatorname{Cap}^{W_N}(V_{N,r})}.
	\end{align}
	where $\sigma_{N,r}^{-2} = 2 \operatorname{Cap}^{W_N}(V_{N,r})$. Plugging this into \eqref{eq:overline_phi_as_Green_function} and then \eqref{eq:overline_phi_weighted_average} gives the claim.
\end{proof}

By the preceding lemma and \eqref{eq:var_overline_Delta_decomp}, we have
\begin{equation}
	\var(\overline \Delta) = \frac{\mathsf{P}^{\gamma \Pi}(\tau_{V_{N,r}} < \tau_{W_N})}{2 \operatorname{Cap}^{W_N}(V_{N,r})} - \frac{1}{2\operatorname{Cap}^{W_N}(V_N)}.
\end{equation}
Now we observe that by time reversal
\begin{align}
	2 \operatorname{Cap}^{W_N}(V_{N,r}) &= \displaystyle  \sum_{x\in \partial^{\text{in}} V_{N,r}} \mathsf{P}^x(\tau_{W_N^\complement} < \tau^+_{V_{N,r}}) \\
	&= |\partial^{\text{out}} W_N| \cdot \mathsf{P}^{\operatorname{Unif}(\partial^{\text{out}} W_N)}(\tau_{V_{N,r}} < \tau^+_{W_N^\complement}) \\
	&= |\partial^{\text{out}} W_N| \cdot  \mathsf{P}^{\operatorname{Unif}(\partial^{\text{out}} W_N)}(\tau_{V_N} < \tau^+_{W_N^\complement}) \\ & \quad \cdot \mathsf{P}^{\operatorname{Unif}(\partial^{\text{out}} W_N)}\big(\tau_{V_{N,r}} < \tau^+_{W_N^\complement}\big|\tau_{V_N} < \tau^+_{W_N^\complement}\big) \\
	&\overset{\eqref{eq:gamma*_random_walk}}{=} 2 \operatorname{Cap}^{W_N}(V_{N}) \, \mathsf{P}^{\gamma^*}(\tau_{V_{N,r}} < \tau^+_{W_N^\complement}).
\end{align}
Thus we have
\begin{align}
	\var(\overline \Delta) &= \frac{\mathsf{P}^{\gamma \Pi}(\tau_{V_{N,r}} < \tau_{W_N^\complement}) - \mathsf{P}^{\gamma^*}(\tau_{V_{N,r}} < \tau_{W_N})}{2 \operatorname{Cap}^{W_N}(V_{N,r})} \\
	&= \frac{\mathsf{P}^{\gamma^*}(\tau_{W_N^\complement} < \tau_{V_{N,r}}) - \mathsf{P}^{\gamma \Pi}(\tau_{W_N} < \tau_{V_{N,r}})}{2 \operatorname{Cap}^{W_N}(V_{N,r})} \lesssim \frac{r}{N}
\end{align}
by Lemma \ref{capacity_bound} and Lemma \ref{lemma:escape_probability_general} (assumptions are satisfied arguing as in the second paragraph following \eqref{eq:kappa} with $U = W_N \setminus V_{N,r}$, $E = \partial^{\text{in}} V_{N,r}$, $D = \partial^{\text{out}} W_N$, and $\Gamma = N V^{r/N}$.

\subsection{Proof of Lemma \ref{lemma:Ding_argument}}\label{s:proof_Ding_argument}

Let $\cC_{8r'}(x)$ be the discrete Euclidean ball of radius $8r'$ centered at $x$, as in \eqref{notation_discrete_ball}, and denote $\mathfrak{C} = \{\cC_{8r'}(x_i): 1\leq i\leq m\}$.  Analogous to \eqref{eq:binding_field_decomposition}, we have the decomposition
\begin{equation}\label{eq:binding_field_decomposition_h'}
	h'(x) = g^Q(x)\oplus \phi'(x) \quad \text{ for all } x \in Q\subset \cC \in \mathfrak{C}
\end{equation}
where $g^Q := h^{\mathcal{C}}|_Q$ is the restriction to $Q$ of $h^{\cC}$ (the DGFF on $\cC$ with zero boundary condition), $\{g^Q: Q\in \cQ_{r'}\}$ independent of each other and of $h'|_{(\cup \mathfrak{C})^\complement}$, and $\phi'$ is the harmonic extension of $h'|_{(\cup \mathfrak{C})^\complement}$ to $\cup \mathfrak{C}$. \medskip

Using the binding field decomposition $g^Q = h^Q \oplus \phi^{\cC, B}$, tightness of $h^Q$ around \eqref{eq:center_max} yields
\begin{equation}
	\bbP(\min g^Q \leq -m_{r'} + C)  \overset{\eqref{conditioning_zero_trick}}{\geq} \frac{1}{2} \bbP(\min h^Q \leq -m_{r'} + C) \geq \frac{1}{4}.
\end{equation}
for some absolute constant $C$. Let $W = \{Q\in \cQ_{r'}: \min g^Q \leq -m_{r'} + C\}$, then by Chernoff bound for Bernoulli r.v.'s
\begin{equation}\label{eq:Ding_eq_14}
	\bbP(|W| \leq m/8) \leq e^{-m/C}. 
\end{equation}
From now on we fix a realization of $\{g^Q: Q\in \cQ_{r'}\}$ for which $|W| > m/8$. In light of \eqref{eq:binding_field_decomposition_h'}, it suffices to prove the following lemma (cf.~\cite[Lemma 2.4]{Ding2013}):

\begin{lemma}\label{Ding_Lem_2.4} Let $U\subset \bigcup \cQ_{r'}$ such that $|U \cap Q| \leq 1$ for all $Q \in \cQ_{r'}$. Assume $|U| \geq m/8$. Then for some absolute constants $C, c > 0$
	\begin{equation}
		\bbP(\min_U \phi' \geq 0 ) \leq e^{-(l/r)/C} + e^{-m/C}.
	\end{equation}
\end{lemma} 

Let us recall the three ingredients used in the proof of \cite[Lemma 2.4]{Ding2013}.

\begin{lemma}\cite[Corollary 2.9]{Ding2013}\label{lemma:Slepian} Let $\{\xi_i: 1\leq i\leq n\}$ be a mean zero Gaussian process such that the correlation coefficients $\rho_{i,j} \leq \rho \leq \frac{1}{2}$ for all $1\leq i<j\leq n$. Then
	\begin{equation}
		\bbP\left(\min_{1\leq i\leq n} \xi_i\geq 0\right) \leq e^{-1/(2\rho)} + (9/10)^n.
	\end{equation}
\end{lemma}

\begin{lemma}\cite[Lemma 6.3.7]{LawlerLimic2010}\label{lemma:Lawler_6.3.7} Let $\cC_n := \{x\in \bbZ^2: \|x\|_2 < n\}$ denote a discrete Euclidean ball of radius $n$ centered at origin. There exist absolute constants $c, C > 0$ such that for all $x\in \cC_{n/4}$ and $y\in \partial^{\text{out}} \cC_n$
	\begin{equation}
		c/n \leq  \mathsf{P}^x(\tau_{\partial^{\text{out}} \cC_n} = y) \leq C/n.
	\end{equation}
\end{lemma}

\begin{lemma}\cite[Proposition 6.4.1]{LawlerLimic2010}\label{lemma:Lawler_6.4.1} For $1 \leq k < n$ and $x\in \cC_n \setminus \cC_k$, we have
	\begin{equation}
	 \mathsf{P}^x(\tau_{\cC_n^\complement} < \tau_{\cC_k}) = \frac{\log \|x\|_2 - \log k + O(1/k)}{\log n - \log k}.
	\end{equation}
\end{lemma}

As a corollary of the last lemma, we have

\begin{corollary}\label{lemma:Lawler_6.4.1_corollary} For $1 \leq k < n $ and $x\in \partial^{\text{out}} \cC_k$, we have
	\begin{equation}
 \mathsf{P}^x(\tau_{\cC_n^\complement} < \tau^+_{\partial^{\text{out}} \cC_k}) = \frac{\Theta(1/k)}{\log n - \log k} .
	\end{equation}
\end{corollary}

We defer the verification of this corollary to the end of this subsection in favor of first giving:

\begin{proof}[Proof of Lemma \ref{Ding_Lem_2.4}] Let $\rho_{x,y}$ be the correlation coefficient between $\phi'(x), \phi'(y)$. By Lemma \ref{lemma:Slepian}, it suffices to show that $\rho_{x, y} = O(r/l)$ for all distinct $x, y \in U$. Applying Lemma \ref{lemma:Lawler_6.3.7} to $x\in Q \subset \cC \in \mathfrak{C}$, we have
	\begin{equation}\label{eq:proof_Ding_Lem_2.4}
		\phi'(x) = \sum_{y\in \partial^{\text{out}} \cC} a_{x, y} h'(y), \quad \text{ where } \frac{c}{r'} \leq a_{x, y} \leq \frac{C}{r'},
	\end{equation}
	which gives
	\begin{equation}\label{eq:proof_Ding_Lem_2.4_var}
		\var( \phi'(x)) = \Theta(1/r'^2) \sum_{y, z\in \partial^{\text{out}} \cC} G_{V_N^-}(y, z) = \Theta(1/r'^2) \sum_{y\in \partial^{\text{out}} \cC} G_{V_N^-}(y, \partial^{\text{out}} \cC).
	\end{equation}
	Now, observing that
	\begin{equation}
		\mathsf{P}^y(\tau_{\cC_{N \diam V}(x)^\complement} < \tau^+_{\partial^{\text{out}} \cC}) \leq \mathsf{P}^y(\tau_{(V_N^-)^\complement} < \tau^+_{\partial^{\text{out}} \cC})  \leq \mathsf{P}^y(\tau_{\cC_{r/2}(x)^\complement} < \tau^+_{\partial^{\text{out}} \cC})
	\end{equation}
	and applying Corollary \ref{lemma:Lawler_6.4.1_corollary} with $k = 8r'$ and $n = N \diam V$ or $n = r/2$, we obtain
	\begin{equation}\label{eq:1/(r'u)}
		\mathsf{P}^y(\tau_{(V_N^-)^\complement} < \tau^+_{\partial^{\text{out}} \cC}) = \Theta(1/(r' u)).
	\end{equation}
	Using \eqref{eq:1/(r'u)} and comparing with geometric r.v.'s, we get
	\begin{equation}\label{eq:Green_function_estimate_for_Ding}
		G_{V_N^-}(y, \partial^{\text{out}} \cC) = \Theta(r'u),
	\end{equation}
	which, together with $|\partial^{\text{out}} \cC| \asymp r'$, yields $\var( \phi'(x) )= \Theta(u)$ in \eqref{eq:proof_Ding_Lem_2.4_var}. Now let us bound the covariances between $\phi'(x)$ and $\phi'(y)$ for distinct $x, y\in U$. By assumption on $U$, $x\in \cC^{(1)}$ and $y\in \cC^{(2)}$ for distinct $\cC^{(1)}, \cC^{(2)} \in \mathfrak{C}$. By \eqref{eq:proof_Ding_Lem_2.4} we see that
	\begin{align}
		\cov(\phi'(x), \phi'(y)) &\leq O(1/r') \max_{w \in \partial^{\text{out}} \cC^{(1)}} G_{V_N^-}(w, \partial^{\text{out}} \cC^{(2)})\\
		&\leq O(1/r') \max_{w\in \partial^{\text{out}}  \cC^{(1)}} \bbP_w(\tau_{\partial^{\text{out}} \cC^{(2)}} < \tau_{(V_N^-)^\complement})\! \max_{z \in \partial^{\text{out}} \cC^{(1)}} G_{V_N^-}(z, \cC^{(2)}) \\
		&\overset{\eqref{eq:Green_function_estimate_for_Ding}}{\leq} O(u) \max_{w\in \partial^{\text{out}}  \cC^{(1)}} \bbP_w(\tau_{\partial^{\text{out}} \cC^{(2)}} < \tau_{(V_N^-)^\complement}).
	\end{align}   
	It remains to show that $\bbP_w(\tau_{\partial^{\text{out}} \cC^{(2)}} < \tau_{(V_N^-)^\complement}) = O(r/l)$ for $w\in \partial^{\text{out}} \cC^{(1)}$ (cf.~\cite[Lemma 2.7]{Ding2013}). Since $\dist(\cC^{(1)}, \cC^{(2)}) \geq l - 2r' $, the preceding probability can be upper bounded by $\bbP_w(\tau_{\cC_{l/2}(w)^\complement} < \tau_{(V_N^-)^\complement})$. This is at most $O(r/l)$ by Lemma \ref{lemma:escape_probability_general} with $U = V_N^- \cap \cC_{l/2}(w)$, $E = \partial^{\text{out}} V_N^- \cap \cC_{l/2}(w)^-$, $D = \partial^{\text{out}} U \setminus E$, and $\Gamma = \bbR^2\setminus N V$.
\end{proof}

Lastly, we give

\begin{proof}[Proof of Corollary \ref{lemma:Lawler_6.4.1_corollary}] The upper bound directly follows from Lemma \ref{lemma:Lawler_6.4.1} because $\mathsf{P}^x(\tau_{\cC_n^\complement} < \tau^+_{\partial^{\text{out}} \cC_k}) \leq \mathsf{P}^x(\tau_{\cC_n^\complement} < \tau_{\cC_k})$. For the lower bound, we first notice that for all $x \in \partial^{\text{out}} \cC_k$, $\|x\|_2 < k+1$ by triangle inequality applied to a neighbor of $x$ in $\cC_k$. Thus, $\partial^{\text{out}} \cC_k\subset \cC_{k+1}$. On the other hand, given any integer $j \geq 1$, there always exists a path from any $x\in \partial^{\text{out}}\cC_k$ to some $$x^{(j)} \in \partial^{\text{out}}\cC_{k+j}$$ with at most $\lceil j \sqrt{2} \rceil$ steps without returning to $\partial^{\text{out}} \cC_k$. (To see this, WLOG $x = (x_1, x_2)$ with $x_1 \geq x_2 \geq 0$, then we can take repeated steps of $(1, 0)$. It is routine to check that $(x_1+1, x_2) \notin \partial^{\text{out}} \cC_k$ and $\|(x_1+s, x_2)\|_2 \geq k + \tfrac{s}{\sqrt{2}}$.) Taking such a path and then applying Lemma \ref{lemma:Lawler_6.4.1}, we have
\begin{align}
\mathsf{P}^x(\tau_{\cC_n^\complement} < \tau^+_{\partial^{\text{out}} \cC_k}) &\geq (\tfrac14)^{\lceil j \sqrt{2} \rceil} \cdot \mathsf{P}^{x^{(j)}}(\tau_{\cC_n^\complement} < \tau_{\cC_{k+1}}) \\
&\gtrsim_j \frac{\log(k+j) - \log(k+1) + O(1/k)}{\log n - \log(k+1)} \\
&\gtrsim_j \frac{\log(1 + j/k) + O(1/k)}{\log n - \log k}
\end{align}
as long as $k + j < n$. By choosing $j$ to be a large absolute constant,  the error term can be absorbed, so the numerator is at least $\Theta(1/k)$ for all $k < n - j$. On the other hand, for $k \geq n - j$, the desired estimate is $\Theta(1)$ and immediately follows from the first line above: in this case $x^{(j)} \in \cC_n^\complement$ already, so the probability on the right side is $1$. 
\end{proof}

\section{Proof of the lower bound in Theorem \ref{main_tail_bound}}\label{s:proof_lower_bound}
\label{s:l}
We use the following (crude) strategy. Consider a scale $r = N e^{-c_0 u}$ for some constant $c_0 > 0$ to be chosen later. As in \eqref{notation_discrete_box}, let $Q_r(x)$ denote the $r\times r$ square centered at $x$. Choose $\mathbb{X}_r \subset (r\mathbb{Z})^2$ such that $\{Q_{r}(x): x\in \mathbb{X}_r\}$ form a minimal covering of $V_N$ by $r\times r$ squares. For each $x \in \mathbb{X}_r$, we consider a decomposition of $h$ into a ``binding field'' $\varphi^{(x)}$ and an (independent) DGFF $h^{(x)}$ on the \emph{enlarged box} $Q_{r'}(x)$ with $r' = r (1 + \delta)$ for some small constant $\delta > 0$:
 \begin{align}
	\varphi^{(x)}(\cdot) &:= \bbE\left[h(\cdot)\Big|h|_{(Q_{r'}(x))^\complement}\right] \\
	h^{(x)}(\cdot) &:= h(\cdot) - \varphi^{(x)}(\cdot). 
\end{align}
Then for each $y \in Q_{r'}(x)$, we have the field decomposition
\begin{align}
	h(y) &= \varphi^{(x)}(y) \oplus h^{(x)}(y)  \\
	 &=\varphi^{(x)}(x)  + (\varphi^{(x)}(y) - \varphi^{(x)}(x)) + h^{(x)}(y).
\end{align}
The right tail event $\Omega_{V_N}(u) := \min_{x\in V_N} h(x) \geq -m_N + u$ can be achieved by the following joint occurrences: In each box $Q_{r'}(x)$,
\begin{itemize}
	\item The ``binding field'' $\varphi^{(x)}$ is close to zero at the center: 
	\begin{equation}\label{eq:lower_bound_condition_1}
		|\varphi^{(x)}(x)| \leq 1.
	\end{equation}
	\item The \emph{downward oscillation} of the ``binding field'' $\varphi^{(x)}$ is sub-linear in $u$: 
	\begin{equation}\label{eq:lower_bound_condition_2}
	\min_{y\in Q_{r}(x)} \varphi^{(x)}(y) - \varphi^{(x)}(x) \geq - u^{\frac{3}{4}} - C
	\end{equation}
	where $C$ is some constant to be chosen later.
	\item The \emph{lift} of the complementary DGFF $h{(x)}$ is high enough: 
	\begin{equation}\label{eq:lower_bound_condition_3}
		\min_{y\in Q_{r}(x)} h^{(x)}(y) \geq -m_N + ku,
	\end{equation}
	where $k > 1$ is some constant to be chosen later.
\end{itemize}

Denote $\wh\bbP(\cdot) := \bbP(\cdot|Z_h=0)$. We abbreviate
\begin{equation}\label{eq:centers_of_boxes_binding_field}
	\vec{Y} := (\varphi^{(x)}(x): x\in \mathbb{X}_r).
\end{equation}
Then the event $\|\vec Y\|_\infty \leq 1$ is precisely the event that \eqref{eq:lower_bound_condition_1} occurs for all boxes. Our first observation is that since each $\varphi^{(x)}(x)$ under $\bbP$ is a Gaussian of variance $\Theta(\log(N/r')) \asymp u$ by Lemma \ref{lemma:magnitude_Delta}, and $\# \mathbb{X}_r \asymp (N/r)^2 = e^{2c_0 u}$, this event should already give the double exponential tail in \eqref{eq:main_tail_bound_lower}. \medskip

On the other hand, we can also show that \eqref{eq:lower_bound_condition_2} and \eqref{eq:lower_bound_condition_3} occur with \emph{overwhelmingly high} probability compared to the number of boxes (provided $k$ is small enough), using Borell-TIS and Fernique (as in the proof of Lemma \ref{oscillation_Delta}) for the binding field $\varphi^{(x)}$ and the tail bound from \cite{Ding-Zeitouni2013} for the complementary DGFF $h^{(x)}$. Formally, we have

\begin{lemma}\label{lem:oscillation_of_varphi_x} There exists constant $C = C(\delta) < \infty$ such that
	\begin{equation}
		\mathbb{P}\left(\exists x\in \mathbb{X}_r: \sup_{y\in Q_r(x)} | \varphi^{(x)} (y)-\varphi^{(x)} (x)| > u^{\frac34} + C \right)\leq 2 e^{-u^{\frac32}/C + 2c_0 u}.
	\end{equation}	
\end{lemma}

\begin{lemma}\label{lem:min_of_h_x} There exists constant $C = C(c_0) < \infty$ such that for all $x \in \mathbb{X}_r$, 
	\begin{equation}
		\bbP\left(\exists x\in \mathbb{X}_r: \min_{y\in Q_{r}(x)} h^{(x)}(y) \leq -m_N + ku\right) \leq C u e^{- (2 c_0 - \sqrt{2\pi} k) u}.
	\end{equation}	
\end{lemma}	

We defer the proofs of these estimates to Section \ref{s:proof_of_lemmas_lower_bound}. Now, for any $c_0 > \sqrt{\pi/2}$, we may choose $k \in (1, 2c_0/\sqrt{2\pi})$. Let $\cA$ be the event that \eqref{eq:lower_bound_condition_2} and \eqref{eq:lower_bound_condition_3} holds for all $x\in \mathbb{X}_r$. Then $\bbP(\cA^\complement) = o(1)$ by the preceding two lemmas, which remains true conditioned on $Z_h = 0$ by \eqref {conditioning_zero_trick} of Lemma \ref{remove_conditioning_trick}:
\begin{equation}\label{eq:overwhelming_probability_unconditional}
	\wh\bbP(\cA^\complement) \leq 2 \bbP(\cA^\complement) = o(1).
\end{equation}
However, since the $o(1)$ here is only exponential decay, we cannot simply use Union Bound. Instead, we will show that $\cA^\complement$ still has $o(1)$ probability \emph{conditioned on a sub-event} of $\|\vec Y\|_\infty \leq 1$ with double exponential probability:

\begin{lemma}\label{lower_bound_main} There exists constant $C < \infty$ such that for $c_0 > \sqrt{\pi/2}$, there exists a set $B \subset \{\|\vec y\|_\infty \leq 1\}$ such that $\wh\bbP(\cA^\complement|\vec Y \in B) = o(1)$ and $\wh\bbP(\vec Y \in B) \geq e^{-C c_0 u e^{2c_0 u}}$.
\end{lemma}

We will prove this lemma in the next section with the help of Lemma \ref{remove_conditioning_trick} and Lemma \ref{conditional_expectation_of_Gaussian}. Let us conclude this section with

\begin{proof}[Proof of Theorem \ref{main_tail_bound}, lower bound] By the discussion above, 
	\begin{align}
		\wh\bbP(\Omega_{V_N}(u)) \geq \wh\bbP(\{\vec Y \in B\} \cap \cA) &= \wh\bbP(\vec Y \in B) \wh\bbP(\cA|\vec Y \in B) \\
		& \geq (1-o(1)) e^{-C c_0 u e^{2c_0 u}}
		\end{align}
where we used Lemma \ref{lower_bound_main}. Finally, for $c > \sqrt{2\pi}$, we can choose $c_0 > \sqrt{\pi/2}$ with $2c_0 < c$, and the inequality above yields the lower bound \eqref{eq:main_tail_bound_lower} for $u \gg 1$. Since $\Omega_{V_N}(u)$ is decreasing in $u$, the lower bound extends to all $u \geq 1$ after increasing $C$.  
\end{proof}

\subsection{Proof of Section \ref{s:proof_lower_bound}'s lemmas}\label{s:proof_of_lemmas_lower_bound}

\begin{proof}[Proof of Lemma \ref{lem:oscillation_of_varphi_x}]Consider $x \in \mathbb{X}_r$. Set $\mu :=\mathbb{E}\left[\sup_{y\in Q_r(x)} | \varphi^{(x)} (y)-\varphi^{(x)} (x)|\right]$ and $\sigma^2:= \sup_{y\in Q_r(x)} \mathrm{Var}[\varphi^{(x)} (y)-\varphi^{(x)} (x)]$. By Borell-TIS, 
	\begin{equation}\label{eq:Borell_TIS_for_Vr}
		\mathbb{P}\left( |\sup_{y\in Q_r(x)} | \varphi^{(x)} (y)-\varphi^{(x)} (x)| -\mu| > t \right)\leq 2 \exp\left(-\frac{t^2}{2\sigma^2}\right). 
	\end{equation}
	To bound $\sigma^2$ we use Lemma \ref{lemma:diff_Delta_squared_modified}. In our case this gives $\sigma^2 \leq C \frac{r}{\delta r} = C \delta^{-1}$. By Lemma \ref{lemma:dudley_modified}, $\mu \leq C' \sqrt{\frac{r}{\delta r}} = C' \delta^{-\frac12}$. Thus, if we choose $t = u^{\frac34}$ in \eqref{eq:Borell_TIS_for_Vr}, we obtain
	\begin{equation}\label{eq:Borell_TIS_for_Vr_applied}
		\mathbb{P}\left( \sup_{y\in Q_r(x)} | \varphi^{(x)} (y)-\varphi^{(x)} (x)| > u^{\frac34} + C' \delta^{-\frac12} \right)\leq 2 \exp\left(-\frac{\delta u^{\frac32}}{2C}\right).
	\end{equation}
The result then follows from Union Bound over $x \in \mathbb{X}_r$.
\end{proof}

\begin{proof}[Proof of Lemma \ref{lem:min_of_h_x}] Recall $r = Ne^{-c_0 u}$, and $m_N = 2\sqrt{g} (\log N - \tfrac{3}{8} \log \log N) + O(1)$ where $g = 2/\pi$, so
	\begin{align}
		-m_N &= -m_r - (m_N - m_r) \\
		&= -m_r - 2\sqrt{g}( c_0 u + \tfrac{3}{8} \log (1 - \tfrac{c_0 u}{\log N}) + O(1)) \\
		&= -m_r -2\sqrt{g} c_0 u + O(1)
	\end{align}
	where we used $u = o(\log N)$. Now we can apply the tail bound from \cite[Theorem 1.4]{Ding-Zeitouni2013}: 
	\begin{align} 
		& \ \bbP\left(\min_{y\in Q_{r}(x)} h^{(x)}(y) \leq -m_N + ku\right) \\
		= & \ \bbP\left(\min_{y\in Q_{r}(x)} h^{(x)}(y) \leq -m_r -(2\sqrt{g} c_0 - k) u + O(1)\right) \\
		\lesssim & \ [(2\sqrt{g} c_0 - k) u \vee 1] \cdot e^{- \sqrt{2\pi}(2\sqrt{g} c_0 -  k) u} \\
		\lesssim & \ c_0 u e^{- (4 c_0 - \sqrt{2\pi} k) u}.
	\end{align}
	provided $(2\sqrt{g} c_0 -  k) u > 1$, else the bound is trivial. The result then follows from Union Bound over $x \in \mathbb{X}_r$.
\end{proof}

\begin{proof}[Proof of Lemma \ref{lower_bound_main}] Let $\Sigma$ and $\wh\Sigma$ be the covariance matrices of $\vec Y$ under $\bbP$ and $\wh\bbP$, respectively.  Consider the set \begin{equation}
		B := \{\|\vec y\|_2 \leq 1, \vec y\cdot \wh\Sigma^\dagger \vec y \leq 1 \}.
	\end{equation}	
	Clearly, $B \subset \{\|\vec y\|_\infty \leq 1\}$. Applying Lemma \ref{remove_conditioning_trick} with 
	\begin{equation}
		\vec X = (h^{(x)}(y), \varphi^{(x)}(y) - \varphi^{(x)}(x): x\in \mathbb{X}_r, y\in Q_{r'}(x))
	\end{equation}
	and then Lemma \ref{conditional_expectation_of_Gaussian} with $H \in \vec X$, we get the first property from \eqref{eq:overwhelming_probability_unconditional}:
	\begin{equation}
		\wh\bbP(\cA^\complement|\vec Y \in B) \leq \frac{\wh\bbP(\cA^\complement)}{\Phi_{\mathcal{N}(0, 1)}(-1)} = o(1). 
	\end{equation}
	It remains to show the second property. Let $\vec{Y}^\dagger := \sqrt{\wh\Sigma^\dagger} \vec Y$, and
	\begin{equation}
		B^\dagger := \{\vec y\cdot \wh\Sigma \vec y \leq 1, \|\vec y\|_2 \leq 1\}.
	\end{equation}
	Then the events $\vec{Y} \in B$ and $\vec{Y}^\dagger \in  B^\dagger$ coincide $\wh\bbP$-a.s. Moreover, bounding the maximal eigenvalue of $\wh\Sigma$ by its trace, we see that $B^\dagger$ must further contain
	\begin{equation}
	B' := \{(\operatorname{Tr}(\wh\Sigma) \vee 1) \|\vec y\|^2_2 \leq 1 \}.
	\end{equation}
	Recall from below \eqref{eq:centers_of_boxes_binding_field} that $\# \mathbb{X}_r \asymp e^{2c_0 u}$ and $\var(\varphi^{(x)}(x)) \asymp \log \# \mathbb{X}_r$, so 
	$$\operatorname{Tr}(\Sigma) \asymp \# \mathbb{X}_r \log \mathbb{X}_r.$$ 
	Conditioning on $Z_h = 0$ only reduces the variances, so $\operatorname{Tr}(\wh\Sigma) \leq \operatorname{Tr}(\Sigma).$ Now, notice that the covariance matrix of $\vec{Y}^\dagger$ is the orthogonal projection onto $\operatorname{Range}(\wh\Sigma)$, so after an orthogonal transformation we may assume the first $\operatorname{rank}(\wh\Sigma)$ entries of $\vec{Y}^\dagger$ are i.i.d.~standard normals and the remaining entries are zeros. Overall, we have
	\begin{align}\label{eq:double_exponential_tail_second}
	\wh\bbP(\vec{Y} \in B) = \wh\bbP(\vec{Y}^\dagger \in  B^\dagger) \geq  \wh\bbP(\vec{Y}^\dagger \in  B') \geq \bbP\left(\sum_{i=1}^{\operatorname{rank}(\wh\Sigma)} U_i^2 \leq \operatorname{Tr}(\Sigma)^{-1}\right)
\end{align}
where $U_i$ are i.i.d.~standard normals, and this can be lower bounded by
\begin{align}
	\bbP\left(U_1^2 \leq \tfrac{\operatorname{Tr}(\Sigma)^{-1}}{\operatorname{rank}(\wh\Sigma)}\right)^{\operatorname{rank}(\wh\Sigma)} \geq e^{-(1+o(1)) \operatorname{rank}(\wh\Sigma) \log \# \mathbb{X}_r } \geq e^{- C c_0 u e^{2c_0 u}}.
\end{align}
\end{proof}

\section{Proofs of auxilary lemmas}\label{s:proofs_preliminaries}
\label{s:a}

First, we give

\begin{proof}[Proof of Lemma \ref{capacity_bound}] By the variational characterization \eqref{def_capacity}, the (discrete) relative capacity is monotonically decreasing when we enlarge the outer domain and shrink the inner domain. Thus, if we choose two concentric discrete balls $B_{in} \subset V_{N}$ with radius $R_{in} \asymp N$ and $B_{out} \supset W_N$ with $R_{out} \asymp N$, then
	\begin{equation}\label{eq:capacity_monotonicty}
		\operatorname{Cap}^{W_N}(V_{N}) \ge \operatorname{Cap}^{B_{out}}(B_{in}).
	\end{equation}	
	Now, we know that (as in the proof of Lemma \ref{lem:sigma_Z_representation}),
	\begin{equation}\label{eq:low_cap}
		2 \operatorname{Cap}^{B_{out}}(B_{in}) = \sum_{z \in \partial^{\text{in}} B_{in}} \mathsf{P}^z(\tau_{\partial^{\text{out}} B_{out}} < \tau^+_{B_{in}}).
	\end{equation}
	By \cite[Proposition~6.4.1]{LawlerLimic2010} each probability in the last sum is uniformly lower bounded by a constant times $ 1/R_{in} \asymp 1/N$. Since $|\partial^{\text{in}} B_{in}| \asymp N$, the sum in \eqref{eq:low_cap} is at least a constant times $N \cdot \frac{1}{N} \asymp 1$. This gives a lower bound for $\operatorname{Cap}^{W_N}(V_{N})$. For the upper bound, we can use the monotonicty property \eqref{eq:capacity_monotonicty} in reverse: fix a closed $C^2$ curve $\Gamma \subset W\setminus V$, let $U$ be the bounded region enclosed by $\Gamma$, and $U_N := N U \cap \mathbb{Z}^2$. Then for large $N$,
	\begin{equation}
		\operatorname{Cap}^{W_N}(V_{N}) \le \operatorname{Cap}^{W_N}(U_N).
	\end{equation}	
	As before we have
	\begin{equation}\label{eq:upper_cap}
		2 \operatorname{Cap}^{W_N}(U_N) = \sum_{z \in \partial^{\text{in}} U_N} \mathsf{P}^z(\tau_{\partial^{\text{out}} W_N} < \tau^+_{U_N}) \leq \sum_{z \in \partial^{\text{out}} U_N} \mathsf{P}^z(\tau_{\partial^{\text{out}} W_N} < \tau_{U_N}).
	\end{equation}
	By Lemma \ref{lemma:escape_probability_general} (with $U = W_N\setminus U_N$, $E = \partial^{\text{in}} U_N$, $D = \partial^{\text{out}} W_N$, and $\Gamma = N U$), each probability in the last sum is upper bounded by a constant times $1/\dist(U_N, W_N^\complement) \asymp 1/N$. Since $\partial U = \Gamma$ is $C^2$, $|\partial^{\text{out}} U_N| \asymp N$, so the sum in \eqref{eq:upper_cap} is at most a constant times $N \cdot \frac{1}{N} \asymp 1$. This gives an upper bound for $\operatorname{Cap}^{W_N}(V_{N})$.
\end{proof}

The remainder of this section is devoted to proving the Lemmas in Section \ref{s:preliminaries}.
\begin{proof}[Proof of Lemma \ref{remove_conditioning_trick}] WLOG we may assume $\vec x = \vec 0$ and $\vec Y$ centered. By Gaussian decomposition, we have
	\begin{align}
		\vec X &= (\vec X|\vec Y = \vec 0) \oplus (\bbE[\vec X|\vec Y] - \bbE[\vec X]) \\
		&= (\vec X|\vec Y = \vec y) \oplus (\bbE[\vec X|\vec Y]-\bbE[\vec X|\vec Y=y])
	\end{align}
	where $\oplus$ denotes independent sum (as usual). The desired result is thus a special case of the following general fact: if $\vec X = \vec X' \oplus \vec X''$, then
	\begin{equation}
		\bbP(\min_i X_i \leq 0) \geq \bbP(\min_i X'_i \leq 0) \min_i \bbP(X''_i \leq 0).
	\end{equation}
	To see this, let $J = \inf\{i: X'_i \leq 0\}$. Then $\min_i X'_i \leq 0$ iff $J < \infty$, in which case $X'_J \leq 0$, and $J$ being a function of $\vec X'$ is independent of $\vec X''$, so
	\begin{align}
		\bbP(\min_i X_i \leq 0) &\geq \bbP(J < \infty, X''_J \leq 0) \\
		&= \sum_{i < \infty} \bbP(J = i, X''_i \leq 0) = \sum_{i < \infty} \bbP(J = i) \bbP(X''_i \leq 0) \\
		& \geq \bbP(J < \infty) \min_i \bbP(X''_i \leq 0) = \bbP(\min_i X'_i \leq 0) \min_i \bbP(X''_i \leq 0).
	\end{align}
\end{proof}

\begin{proof}[Proof of Lemma \ref{conditional_expectation_of_Gaussian}] W.l.o.g. we may assume $X$ is centered and $\bbE[X|\vec Y]$ is non-degenerate. By Gaussianity, $\bbE[X|\vec Y] = \vec{l}\cdot \vec Y$ for some $\vec l \neq \vec 0$. Then
	\begin{align}
		\bbE[X|\vec Y = \vec y] = \vec{l}^T \vec y
		= \vec{l}^T \sqrt{\Sigma} \cdot \sqrt{\Sigma^\dagger} \vec y
		\geq -\sqrt{\vec{l}\cdot \Sigma \, \vec{l}} \cdot \sqrt{\vec{y}\cdot \Sigma^\dagger \vec y}
	\end{align}
	by Cauchy-Schwarz inequality, where
	\begin{equation}
		\vec{l}\cdot \Sigma \, \vec{l} = \var(\vec l\cdot \vec Y) = \var (\bbE[X|\vec Y]).
	\end{equation}
	Thus,
	\begin{equation}
		\bbP\left(\bbE[X|\vec Y] \leq \bbE[X|\vec Y = \vec y]\right) \geq \bbP\left(\bbE[X|\vec Y] \leq -\sqrt{\var (\bbE[X|\vec Y])}  \cdot \sqrt{\vec{y}\cdot \Sigma^\dagger \vec y}\right)
	\end{equation}
	and the result follows after normalizing $\bbE[X|\vec Y]$ by its standard deviation.
\end{proof}

\begin{proof}[Proof of Lemma \ref{lemma:magnitude_Delta}] This is a standard computation of Green functions. By the Gibbs-Markov decomposition $h^{W} = h^{V} \oplus \phi^{W,V}$, we have
	\begin{equation}
		\var(\phi^{W, V}(x)) = \var(h^W(x)) - \var(h^V(x)) = G_{W}(x, x) - G_{V}(x, x) 
	\end{equation}
	By \cite[Proposition~4.6.2]{LawlerLimic2010} and \cite[Theorem~4.4.4]{LawlerLimic2010},
	\begin{align}
		G_{W}&(x, x) - G_{V}(x, x)  = \sum_{\partial^{\text{out}} W} \Pi_{W}(x, w) a(x-w) - \sum_{\partial^{\text{out}} V} \Pi_{V}(x, v) a(x-v) \\
		&\leq g \max_{w\in \partial^{\text{out}} W} \log \|x-w\|_2 - g \min_{w\in \partial^{\text{out}} V} \log \|x-v\|_2 + O(\dist( x, \partial^{\text{out}} V)^{-2})\\
		&\leq g \log (\diam W) - g\log (\dist(x, V^\complement)) + O(\dist(x, V^\complement)^{-1})
	\end{align}
	where $a$ denotes the potential kernel on $\mathbb{Z}^2$ and $\Pi_A(x, \cdot)$ denotes the harmonic measure from $x$ in $A$. This gives the upper bound. The lower bound is analogous. 
\end{proof}

\begin{proof}[Proof of Lemma \ref{lemma:diff_Delta_squared_modified}]
	By the Gibbs-Markov decomposition $h^{W} = h^{V} \oplus \phi^{W,V}$, 
	\begin{align}\label{eq:3.4}
		& \mathbb{E}\left[\left(\phi^{W, V}(x)-\phi^{W, V}(y)\right)^2\right]  = \mathbb{E}\left[\left(h^{W}(x)-h^{W}(y)\right)^2\right] - \mathbb{E}\left[\left(h^{V}(x)-h^{V}(y)\right)^2\right]
		\\
		&= (	G_{W}(x,x)- G_{W}(x,y)) +	(G_{W}(y,y)- G_{W}(x,y))\\
		&\quad - (	G_{V}(x,x)- G_{V}(x,y))- (	G_{V}(y,y)- G_{V}(x,y)).
	\end{align}
	By \cite[Proposition~4.6.2]{LawlerLimic2010} and \cite[Theorem~4.4.4]{LawlerLimic2010},
	\begin{align}
		G_{W}(x,x)&- G_{{W}}(x,y) = a(x-y) + \sum_{z\in \partial^{\text{out}} {W}}\Pi_{W}(x,z) \left(a(x-z)-a(y-z)\right) \\
		&= a(x-y) + g \sum_{z\in \partial^{\text{out}} {W}}\Pi_{W}(x,z) \log \left(\frac{\|x-z\|_2}{\|y-z\|_2}\right) + O(\dist( U, W^\complement)^{-2}) \label{eq:3.6}
	\end{align}
	where $a$ denotes the potential kernel on $\mathbb{Z}^2$ and $\Pi_W(x, \cdot)$ denotes the harmonic measure from $x$ in $W$. Using the inequality $\log(1+x) \leq x$, we have for any $z\in \partial^{\text{out}} {W}$
	\begin{equation}\label{eq:3.7}
		\log \left(\frac{\|x-z\|_2}{\|y-z\|_2}\right) \leq \log\left(\frac{\|y-z\|_2+\|y-x\|_2}{\|y-z\|_2}\right)  
		\leq \frac{\|y-x\|_2}{\dist(U,W^\complement)}.
	\end{equation}
	This shows
	\begin{equation}
		G_{W}(x,x)- G_{{W}}(x,y) \leq a(x,y) +  g \frac{\|y-x\|_2}{\dist(U, W^\complement)} +O(\dist(U, W^\complement)^{-2}).
	\end{equation}
	and the same bound holds for $G_{W}(y,y)- G_{{W}}(x,y)$ by switching $x$ and $y$. Next, we bound \emph{from below} $G_{{V}}(x,x) - G_{{V}}(x,y)$. We use \eqref{eq:3.6} with $V $ in place of $W$, and rewrite
	\begin{equation}
		\log \left(\frac{\|x-z\|_2}{\|y-z\|_2}\right) = - \log \left(\frac{\|y-z\|_2}{\|x-z\|_2}\right)
	\end{equation}
	so the (positive) upper bound in \eqref{eq:3.7} yields a (negative) lower bound after switching $x$ and $y$.  
	Hence,
	\begin{equation}
		G_{{V}}(x,x) - G_{{V}}(x,y) \geq a(x-y) -  g\frac{\|x-y\|_2}{\dist(U, V^\complement)} +O(\dist(U, V^\complement)^{-2})
	\end{equation}
	and the same lower bound holds for $G_{V}(y,y)- G_{{V}}(x,y)$ by switching $x$ and $y$.
	Lastly, note that $\dist(U, V^\complement) \leq \dist(U, W^\complement)$ and all occurrences of $a(x, y)$'s cancel.
	Collecting everything and plugging into \eqref{eq:3.4} concludes the proof.
\end{proof}

\begin{proof}[Proof of Lemma \ref{lemma:dudley_modified}] This follows from Fernique's majorizing theorem with the normalized counting measure on $U$. In more details, let $\rho(x,y) := \sqrt{\mathbb{E}[(\Delta(x)-\Delta(y))^2]}$ denote the corresponding natural intrinsic metric. Let $B_{\rho} (x,r):=\{y\in U: \rho(x,y)<r\}$ and $\cC_r(x):=\{y\in \mathbb{Z}^2: \|x-y\|_2<r\}$. By assumption \eqref{eq:dudley_assumption_2}, $\diam_{\rho}(U) \leq \sqrt{L \diam(U)}$ and
	\begin{equation}
		B_{\rho}(x, \sqrt{Lr}) \supset \cC_r(x) \cap U,
	\end{equation}
	for $x\in U$. This implies
	\begin{equation}
		| B_{\rho}(x,r)| \geq | \cC_{L^{-1} r^2}(x) \cap U| \geq c_U L^{-2}r^4
	\end{equation} 
	for $x\in U, r\leq \diam_\rho (U)$. Applying Fernique's majorizing theorem with majorizing measure the normalized counting measure on $U$,
	\begin{align}
		\mathbb{E}\left[\sup_{x,y\in U} | \Delta (x)-\Delta(y)|\right]&\leq K \int_{0}^{\diam_\rho (U)}\sqrt{\log \left(\frac{1}{|B_{\rho}(x,r)| /| U|}\right)} \mathrm{d}r\\
		& \leq K \int_{0}^{\diam_\rho (U)}\sqrt{\log \left(\frac{1}{c_U L^{-2} r^4 /|U|}\right)} \mathrm{d}r \\
		&\leq K (| U| \cdot L^2 /c_U)^{1/4} \int_{0}^{1}\sqrt{\log \left(1/r^4\right)} \mathrm{d}r \\
		&\lesssim K c_U^{-\tfrac 14}\sqrt{L \diam(U)}
	\end{align} 
	where we applied a rescaling in $r$, that $r\to \sqrt{\log(1/r^4)}$ is integrable near $0$, and that $| U| \lesssim \diam(U)^2$.
\end{proof}

\begin{proof}[Proof of Lemma \ref{lemma:escape_probability_general}] First we may assume 
	\begin{equation}\label{eq:escape_non_degenerate}
		0 < \dist(x, E) \leq \dist(x, D)/4, \text{ and } \dist(x, D) \geq 2K M
	\end{equation}
	for a large constant $K \geq 6$ to be chosen later, otherwise the estimate is trivial. The curvature bound \eqref{eq:Gamma_assumption_curvature} on $\partial \Gamma$ implies a \emph{uniform interior ball condition} for the region $\Gamma$  (see e.g. \cite[Section~14.16]{GilbargTrudinger2001}). Specifically, we choose a radius
	\begin{equation}\label{eq:radius_of_curvature} 
		r = \min\left(\frac{1}{\kappa_{\max}}, \frac{\dist(x, D)}{K}\right) \geq 2M.
	\end{equation}
	This choice ensures two geometric properties. First, at any point on $\partial \Gamma$, there exists a ball of radius $r$ tangent to $\partial \Gamma$ completely contained inside $\Gamma$. Secondly, the radius is small enough compared to the distance to the escape set $D$.
		
	Take $x' \in E$ closest to the starting point $x$, i.e. $\|x - x'\|_2 = \dist(x, E)$. Then take $y\in \partial \Gamma$ closest to $x'$, so $\|y - x'\|_2 = \dist(x', \partial \Gamma) \leq M$ by assumption \eqref{eq:Gamma_assumption_1}. Let $\nu(y)$ be the unit normal vector at $y$ pointing \textit{into} the region $\Gamma$, and define
	\begin{equation} 
		z := y + r \nu(y).
		\end{equation}
	 Then $z \in \mathbb{R}^2$ is the center of a ball of radius $r$ tangent to $\partial \Gamma$ at $y$ completely contained inside $\Gamma$. We deduce the following geometric bounds:  
	\begin{itemize}
		\item By triangle inequality
		\begin{align}\label{eq:inner_dist}
			\dist(z, E) \leq \|z - y\|_2 + \|y - x'\|_2 \leq r + M,
		\end{align}
		and by assumption \eqref{eq:Gamma_assumption_1}
		\begin{align}\label{eq:inner_dist_2}
			\dist(z, E) \geq \dist(z, \partial \Gamma) - \max_{w\in E} \dist(w, \partial \Gamma) \geq r - M.
		\end{align}
		\item By triangle inequality
		\begin{align}\label{eq:start} 
			\|x - z\|_2 &\leq \|x - x'\|_2 + \|x' -y\|_2 + \|y - z\|_2 \\
			&\leq \dist(x, E) + M + r,
		\end{align}
		and by the second part of assumption \eqref{eq:Gamma_assumption_1},
		\begin{equation}\label{eq:start_lower} 
			\|x-z\|_2 \geq \dist(z, \Gamma^\complement) - \dist(x, \Gamma^\complement) \geq r - M.
		\end{equation}
		\item By (reverse) triangle inequality and \eqref{eq:start},
		\begin{align}\label{eq:outer_dist}
		\dist(z, D) &\geq \dist(x, D) - \|x-z\|_2 \\
		&\geq \dist(x, D)  - r - M - \dist(x, E)\geq \dist(x, D)/2, 
 			\end{align}
		where we used the conditions \eqref{eq:escape_non_degenerate} and \eqref{eq:radius_of_curvature}.
	\end{itemize}
	
	Let $\tilde{z} \in \mathbb{Z}^2$ be a lattice point closest to $z$. By \eqref{eq:start_lower} applied over all $x\in U$,
	\begin{equation}
		\dist(z, U) \geq r - M \overset{\eqref{eq:radius_of_curvature} }{\geq} M \geq 1,
	\end{equation}
	which shows $\tilde{z}$ is outside $U$. We utilize the potential kernel $a: \mathbb{Z}^2 \to \mathbb{R}$ for the simple random walk.
	According to \cite[(6.11)]{LawlerLimic2010}, the potential kernel satisfies the asymptotic expansion
	\begin{equation}\label{eq:a_asymp}
		a(u) = \tfrac{2}{\pi} \log \|u\|_2 + \gamma_2 + O(\|u\|_2^{-2}) \quad \text{as } \|u\|_2 \to \infty,
	\end{equation}
	where $\gamma_2$ is some constant. Furthermore, by \cite[Proposition~4.4.2]{LawlerLimic2010}, $a$ is harmonic on $\mathbb{Z}^2 \setminus \{0\}$.
	Define the function $\tilde a (u) := a(u - \tilde{z})$. Since $\tilde{z}$ is outside $U$, $\tilde a$ is harmonic for the random walk up to the stopping time $\tau_{U^\complement} = \tau_{D} \wedge \tau_{E}$.
	Let $p := \mathsf{P}^x(\tau_{D} < \tau_{E})$. By the Optional Stopping Theorem
	\begin{equation} 
		\tilde a(x) = \mathbb{E}^x[\tilde a(S_{\tau_{U^\complement} })] \ge p \min_{w \in D} \tilde a(w) + (1-p) \min_{v \in E} \tilde a(v). 
	\end{equation}
	Rearranging and using the asymptotics \eqref{eq:a_asymp} we get
	\begin{equation}
		p \leq \frac{\log\left(\displaystyle\frac{\|x-\tilde{z}\|_2}{\dist(\tilde z, E)}\right)  + \displaystyle\max_{u\in \{x\} \cup E \cup D} O(\|u-\tilde{z}\|_2^{-2})} {\log\left(\displaystyle\frac{\dist(\tilde z, D)}{\dist(\tilde z, E)}\right)}.
	\end{equation}
 After replacing $\tilde{z}$ with $z$ and accounting for the resulting $O(1)$ discretization errors inside the logarithms, the geometric bounds \eqref{eq:inner_dist}, \eqref{eq:inner_dist_2}, \eqref{eq:start}, \eqref{eq:start_lower}, and \eqref{eq:outer_dist} give us
	\begin{equation}
		p \leq \frac{\log\left(1 + \tfrac{\dist(x, E)}{r}\right) + O(\tfrac{M}{r})}{\log\left(\frac{\dist(x, D)}{2r}\right) + O(\tfrac{M}{r})}.
	\end{equation}
	By the inequality $\log(1+t) \leq t$, the numerator is at most $\tfrac{\dist(x, E)+O(M)}{r}$. In the denominator, the assumption \eqref{eq:escape_non_degenerate} implies $\log\left(\frac{\dist(x, D)}{2r}\right)$ is at least $\log (K/2)$, while the error term is at most $O(M/2M) = O(1)$. By choosing $K$ to be a large absolute constant, the error term may be absorbed at the cost of a $1/2$ factor in the denominator. Overall we have
	\begin{align}
		p&\leq \frac{\dist(x, E)+O(M)}{r \log\left(\frac{\dist(x, D)}{2r}\right)/2} \\
		&\lesssim_M \frac{\dist(x, E)}{\dist(x, D)} \cdot \frac{\tfrac{\dist(x, D)}{2r}}{\log\left(\frac{\dist(x, D)}{2r}\right)},
	\end{align}
	where $\lesssim_M$ means the implicit constant depends only on $M$. It remains to bound the last factor, which is of the form $s/\log s$ with $s = \tfrac{\dist(x, D)}{2r}$. By our choice of $r$ in \eqref{eq:radius_of_curvature}, $s = \max(K/2, \kappa_{\max} \tfrac{\dist(x, D)}{2})$. Since $\dist(x, D) \leq \diam(U)$ and $\kappa_{\max} \diam(U)\leq R$ by assumption \eqref{eq:Gamma_assumption_curvature}, $K/2 \leq s \leq \max(K/2, R/2)$, and $s/\log s$ is bounded by a constant depending on $R$.
\end{proof}

\section*{Acknowledgements}
We thank Naomi and Ohad Feldheim for very useful discussions, ultimately leading to the decomposition in~\eqref{e:1.11}.
The research of O.L., M.F. and T.W were supported by ISF grants no.~2870/21 and~3782/25.
The research of T.W. was also supported by a fellowship from the Lady Davis Foundation.
The research of M.F. was also supported by the Deutsche Forschungsgemeinschaft (DFG) grant  DR 1096/2-1.

\setlength\parskip{0pt}
\bibliographystyle{abbrv}
\bibliography{bib} 
\end{document}